\documentclass{imsart}

\usepackage{amsmath}
\usepackage{amsthm}
\usepackage{amssymb}

\usepackage{mhequ}
\usepackage{graphicx}
\usepackage{placeins}
\usepackage{xfrac}
\usepackage{times}
\usepackage{mathrsfs}
\usepackage[numbers]{natbib}
\usepackage[colorlinks]{hyperref} 

\startlocaldefs
\theoremstyle{plain}

\newtheorem{theorem}{Theorem}[section]

\newtheorem{proposition}[theorem]{Proposition}
\newtheorem{corollary}[theorem]{Corollary}

\newtheorem{assumption}[theorem]{Assumption}
\newtheorem{remark}[theorem]{Remark}
\newtheorem{lemma}[theorem]{Lemma}

\renewcommand{\P}{\mathcal P}
\newcommand{\X}{\mathbf X}
\newcommand{\ccdot}{\,\cdot\,}
\newcommand{\be}{\begin{equs}}
\newcommand{\ee}{\end{equs}}
\newcommand{\vertiii}[1]{{\left\vert\kern-0.25ex\left\vert\kern-0.25ex\left\vert #1 
    \right\vert\kern-0.25ex\right\vert\kern-0.25ex\right\vert}}
\newcommand{\lvertiii}{{\vert\kern-0.25ex\vert\kern-0.25ex\vert}}    
\newcommand{\rvertiii}{{\vert\kern-0.25ex\vert\kern-0.25ex\vert}}
\newcommand{\1}{\mathbf 1}    
\newcommand{\R}{\mathbb R}

\DeclareMathOperator{\Normal}{N}
\DeclareMathOperator{\tr}{trace}

\DeclareMathOperator{\var}{var}
\DeclareMathOperator{\KL}{KL}

\DeclareMathOperator{\argmin}{argmin}
\DeclareMathOperator{\Lip}{Lip}
\DeclareMathOperator{\diag}{diag}

\newcommand{\mc}[1]{\mathcal #1}
\newcommand{\bigO}{\mathcal O}

\newcommand{\E}{\mathbf E}
\DeclareMathOperator{\Binom}{Binomial}

\endlocaldefs
    
\graphicspath{{./../Code/Figures/}}    

\begin{document}
\begin{frontmatter}

\title{Error bounds for Approximations of Markov chains used in Bayesian Sampling}
\runtitle{Error bounds for Approximations of Markov chains}

\author{\fnms{James E.} \snm{Johndrow}\corref{} \ead[label=e1]{johndrow@stanford.edu}}
\address{390 Serra Mall \\ Stanford, CA, USA \\ \printead{e1}}
\affiliation{Stanford University}
\and
\author{\fnms{Jonathan C.} \snm{Mattingly}\ead[label=e2]{jonm@math.duke.edu}}
\address{Departments of Mathematics and Statistical Science\\ Duke University \\ Durham, NC, 28808 \\ \printead{e2}}
\affiliation{Duke University}

\runauthor{Johndrow and Mattingly}

\begin{abstract}
We give a number of results on approximations of Markov kernels in total variation
and Wasserstein norms weighted by a Lyapunov function. The results are applied to
examples from Bayesian statistics where approximations to transition kernels 
are made to reduce computational costs. 
\end{abstract}

\begin{keyword}[class=MSC60J05]
\kwd{Markov chain Monte Carlo}
\kwd{Martingale methods}
\kwd{Lyapunov function}
\kwd{Wasserstein metric}
\end{keyword}

\end{frontmatter}


\maketitle
Often in modern computational Bayesian statistics one approximates a
Markov chain used to estimate posterior exceptions with a
computationally simpler Markov chain. Such approximations 
are growing in popularity due to the prevalence of high dimensional
parameter spaces and large sample size in contemporary applications.
Some examples can be found in \cite{korattikara2014austerity,
welling2011bayesian,chen2014stochastic, banerjee2008gaussian, bardenet2014towards}.
Motivated by this, we study a number of approximation schemes and 
develop general approximation results to control the error introduced 
by using an approximating Markov chain. 

We do not strive for complete generality but rather study a few
examples which capture the essence of many approximation
schemes. We emphasize a unified presentation which turns on
establishing approximation error in a metric -- or in some cases a semimetric
that does not satisfy the triangle inequality -- for which the original
Markov kernel is a strict contraction. This greatly simplifies the
presentation relative to previous works. It also works equally well in
total variation, Wasserstein metrics, and semimetric notions of convergence.   

We begin with the setting where the approximating and exact kernels
are absolutely continuous with respect to Lebesgue measure. Our motivating
examples are Gaussian approximations in Gibbs sampling and minibatching 
Metropolis-Hastings. In this setting it is natural to study convergence
in a weighted total variation norm.


We then consider 
Metropolis-Hastings schemes for sampling the posterior in Gaussian process models that
arise in spatial statistics and nonparametric regression. As
commonly done in the literature, we employ approximation schemes
which make use of low-rank approximations of a covariance matrix 
and discretization of the state space. These
produce approximating transition kernels which are singular (in the sense of
measure) to the original transition kernel. Hence, the approximations
are not close in the total variation norm. However, we will show that
they are close in the 1-Wasserstein norm. As the original kernel is not 
a strict contraction in 1-Wasserstein in some examples, we also consider
weighting by the Lyapunov function, resulting in a semimetric introduced in
\cite{hairer2011asymptotic}. 

Many of the results we give in total variation
apply equally in Wasserstein norms. The application-driven nature of the
presentation notwithstanding, the motivation for 
treating the Wasserstein and total variation cases separately is that the most
difficult part of applying our results is showing that the original kernel
is a strict contraction in the selected metric/semimetric. These proofs are fairly standard
in total variation norms and are omitted, but are less frequently seen in 
Wasserstein norms. We therefore devote some time to giving verifiable conditions
for a kernel to be a strict contraction in Wasserstein metrics -- and Wasserstein-like 
semimetrics -- including adapting some arguments from
the continuous time case to our discrete time setting. These conditions are then 
verified for our motivating applications.

Our development is inspired by \cite{hairer2011yet} and \cite{hairer2011asymptotic}
in its focus on proving strict contractions in weighted supremum norms for the 
original kernel. 
In essence, we show that if one begins with an exact kernel that is
a strict contraction, it is quite easy to obtain useful bounds on the approximation error
of time averages and other pathwise quantities. Consequently, the main effort in showing
that an approximation scheme satisfies our conditions is often showing that the original
kernel is a strict contraction. We therefore devote some time to giving verifiable conditions
for a kernel to be a strict contraction, particularly in the less standard weighted Wasserstein-like 
semimetric, where we give a finite-time version of the result in \cite{hairer2011asymptotic}.

Several of our results are directly comparable to those in 
\cite{rudolf2015perturbation}, and, to some extent, \cite{pillai2015ergodicity}.
\citet{rudolf2015perturbation} in particular also prove perturbation bounds in 
weighted supremum norms, but we choose certain constants as in
\cite{hairer2011asymptotic, hairer2011yet} to obtain 
strict contractions, simplifying many of the arguments. The results using non-metric
notions of convergence are unlike \cite{pillai2015ergodicity} or \cite{rudolf2015perturbation}. 
We also provide variation
bounds using Poisson equation techniques similar to \cite{GlynnMeyn1996,KontoyiannisMeyn2003,KontoyiannisMeyn2012,MattinglyStuartTretyakov2015}. 
Of course, such methods are intimately related to classical Martingale and potential methods 
\cite{MR0415773} as
well as classical ideas from dynamical systems. If no other structure
is assumed, good control of the approximation error for 
time averages turns out to require the pointwise approximation error to be small relative
to the ``spectral gap'' in the chosen metric/semimetric. This resembles the situation for uniformly ergodic chains, 
studied by the authors in \cite{johndrow2017coupling}, and others in \cite{mitrophanov2005sensitivity, 
alquier2014noisy}. 

\vspace{1em}
\noindent {\bf Acknowledgements.} Both authors thank the NSF for its
partial suport of this project through the grant DMS-3332219 as well
as David Dunson, Sayan Mukherjee, Bamdad Hosseini, and Daniel Rudolf 
for useful and stimulating conversations.

\section{Bounds in weighted total variation} \label{sec:Bounds}
We begin by defining the weighted total variation metrics which will
be used to quantify convergence, largely following
\cite{hairer2011yet} which was informed by \cite{GlynnMeyn1996,meyn2012markov}. 
Many of the results will actually hold for any metric, which
we point out as appropriate. We then introduce a approximating change and
bound the shift in the invariant measure and time averages. 

\subsection{Basic mixing results}
We assume conditions on the Markov kernel similar to those in \cite{hairer2011yet,meyn2012markov}.
Let $\P(x,\ccdot)$ be a Markov kernel on a Polish state space $\mathbf X$, which in many
applications is $\mathbb R^p$, $p$-dimensional Euclidean space. We use $\P$ for operators
defined on the set of measurable functions and the set of finite measures
\be
(\P \varphi)(x) = \int_{\X} \varphi(y) \P(x,dy), \quad (\mu \P)(A) = \int_{\X} \P(x,A) \mu(dx).
\ee
We assume that $\P$ satisfies a Foster-Lyapunov drift condition

\begin{assumption} \label{ass:lyapunov}
  There exists a function $V: \X \to [0,\infty)$ such that for some $\gamma \in (0,1)$ and $K >0$
\begin{align}
  \label{eq:4}
  (\P V)(x) \leq \gamma V(x) + K
\end{align}
for all $x \in \X$.
\end{assumption}

We also assume that sublevel sets of $V$ are ``small'' in that they satisfy a uniform
minorization condition.
\begin{assumption} \label{ass:small}
  For every $R>0$, there exists $\bar \alpha_0 \in (0,1)$ (depending on $R$) such that
\begin{align}
  \label{eq:3}
\sup_{x,y \in \mathcal S(R)}  d_0(\delta_x \P, \delta_y \P) \leq 2\bar \alpha_0
\end{align}
for $\mathcal S(R)=\{ x : V(x) \leq R\}$, where $d_0$ is the total variation metric. 
\end{assumption}

To quantify the rate of convergence to equilibrium, we procede in the
spirit of \cite{meyn2012markov} and  define a family of weighted supremum norms indexed by a scale parameter $\beta > 0$
by
\be
\| \varphi \|_\beta = \sup_x \frac{|\varphi(x)|}{1+\beta V(x)}
\ee
and the dual metric $\rho_\beta$ on probability measures
\be \label{eq:rho}
\rho_\beta(\mu,\nu) = \sup_{\varphi: \|\varphi\|_\beta \le 1} \int_{\X} \varphi(x) (\mu - \nu)(dx) = \int_{\X} (1+\beta V(x))|\mu-\nu|(dx),
\ee
a weighted total variation distance.
\citet{hairer2011yet} show that for $\beta$ sufficiently small, the
Markov semigroup $\P$ is a contraction in the metric $\rho_\beta$
under Assumptions~\ref{ass:lyapunov} and~\ref{ass:small}. In
\cite{hairer2011yet}, they also showed 
 that these metrics are equivalent to the metric
$d_\beta$ on measures induced by
\be \label{eq:dBeta}
d_\beta(x,y) = \left \{ \begin{array}{cc} 0 & x=y \\ 2 + \beta V(x) + \beta V(y) & x \ne y \end{array} \right.\,
\ee
To define $d_\beta$ for measures, one first defines a Lipschitz seminorm on
measurable functions by
\be \label{eq:LipschitzNorm}
\vertiii{\varphi}_\beta &= \sup_{x \ne y} \frac{|\varphi(x) -
  \varphi(y)|}{d_\beta(x,y)}\,.
\ee
This in turn induces the metric $d_\beta$ on probability measures
through
\be \label{eq:DualMetric}
\hspace{-2mm}d_\beta(\mu,\nu) &= \sup_{\varphi : \vertiii{\varphi}_\beta \le 1} \int_{\X} \varphi(x)(\mu-\nu)(dx) = \inf_{\Gamma \in \mathcal C(\mu,\nu)} \int_{\X \times \X} d(x,y) \Gamma(dx,dy)
\ee
for which it turns out that $\vertiii{\varphi}_\beta = \inf_{c \in \mathbb R} \|\varphi + c\|_\beta$ and therefore 
$d_\beta = \rho_\beta$. Here $\mathcal C(\mu,\nu)$ is the space of all couplings of $\mu,\nu$. 
In the sequel we freely interchange $d_\beta$ and $\rho_\beta$.
We now give the convergence theorem from
\citet{hairer2011yet} which uses these metrics.
\begin{theorem}[Theorem 1.3 of \cite{hairer2011yet}] \label{thm:HM2011}
  Under Assumptionss~\ref{ass:lyapunov} and~\ref{ass:small}, there
  exist an $\bar \alpha \in (0,1)$ and $\beta >0$ so that 
  \begin{align*}
    d_\beta(\nu_1 \P, \nu_2 \P) \leq \bar \alpha d_\beta(\nu_1,\nu_2) 
  \end{align*}
for all probability measure $\nu_1$ and $\nu_2$.
\end{theorem}

\subsection{Basic approximation results}

Now consider a second transition kernel $\P_\epsilon$ that is ``nearby'' $\P$ in the following sense.
\begin{assumption} \label{ass:closeness}
  For some $\delta \geq 0$ and all $x$,
\begin{align*}
  d_1(\delta_x \P, \delta_x \P_\epsilon) &\leq \epsilon (1+\delta
  V(x))\\ & \Updownarrow \\ (\P-\P_\epsilon)\varphi(x) \leq
  \epsilon(1+\delta V(x)) &\text{ for all $|\varphi| \leq 1 + V$} 
\end{align*}
\end{assumption}

The following basic pertubation bound is one of our main results.
\begin{theorem} \label{thm:TVperturb}
  Suppose assumptions \ref{ass:lyapunov}, \ref{ass:small}, and \ref{ass:closeness} hold. 
  Then there exists a $\beta \in (0,1]$ and $\bar \alpha \in (0,1)$ so that
  \begin{align*}
    \P_\epsilon \varphi(x) -     \P \varphi(y)  \leq \epsilon (1+ \delta V(x))
    +  \bar \alpha \, d_\beta(x,y) 
  \end{align*}
for all $|\varphi| \leq 1 + \beta V$.
\end{theorem}
\begin{proof}
We have
\be
d_\beta(\delta_y \P,\delta_x \P_\epsilon) &\le d_\beta(\delta_y \P,\delta_x \P) + d_\beta(\delta_x \P, \delta_x \P_\epsilon) \\
&\le \bar \alpha d_\beta(x,y) + \epsilon (1+\delta V(x)),
\ee
where the first term followed from Assumptions \ref{ass:lyapunov} and \ref{ass:small} and \cite[Theorem 3.1]{hairer2011yet} and the second term from Assumption \ref{ass:closeness}. 
\end{proof}
As we are about to see, this inequality will be sufficient for many purposes.
Thus, if one is careful about defining metrics to obtain strict contraction, 
perturbation bounds can be obtained simply from the triangle inequality.
For example, this immediately gives a bound on the distance between the one-step
transition kernels for any pair of starting measures.
\begin{corollary}
  Let $\mu$ and $\nu$ be two probability measures. Then 
  \be \label{eq:DistanceInvariant}
    d_\beta(\mu \P_\epsilon, \nu \P) \leq \epsilon ( 1 + \delta \mu V) +
    \bar \alpha  d_\beta(\mu , \nu).
  \ee
\end{corollary}

We can use this result to bound the distance between the invariant
measure(s).  If $\mu_0$ and $\mu_\epsilon$ are invariant 
measures of $\P$ and $\P_\epsilon$ respectively then
\be
   d_\beta(\mu_\epsilon , \mu_0) \leq   \frac{\epsilon}{1-\bar \alpha} ( 1 +
  \delta \, \mu_\epsilon V ).
\ee
The ratio of the pointwise approximation error $\epsilon$ to the spectral
gap $1-\bar \alpha$ is a key quantity that will appear often; clearly
$\epsilon \ll 1-\bar \alpha$ implies small bias.
Iterating the estimate in \eqref{eq:DistanceInvariant} gives
\be \label{eq:FiniteTime}
    d_\beta(\mu \P_\epsilon^n, \nu \P^n) \leq \epsilon  \sum_{k=1}^{n} \bar \alpha^{n-k}( 1 + \delta 
  \mu \P_\epsilon^{k-1} V) +
    \bar \alpha^n  d_\beta(\mu , \nu),
\ee
a finite-time error bound. If we now assume
\begin{assumption} \label{ass:lyapunovPeps}
  For some $\gamma_\epsilon \in (0,1)$ and $K_\epsilon >0$
\begin{align}
  \label{eq:e4}
  (\P_\epsilon V)(x) \leq \gamma_\epsilon V(x) + K_\epsilon
\end{align}
for all $x$.
\end{assumption}
so that $V$ is also a Lyapunov function of $\P_\epsilon$, then
\begin{align*}
   \mu \P_\epsilon^{j} V \leq \gamma^j_\epsilon \mu V + \frac{K_\epsilon}{1-\gamma_\epsilon},
\end{align*}
and in place of \eqref{eq:FiniteTime} we can use the bound
\be
  d_\beta(\mu \P_\epsilon^n, \nu \P^n) &\leq
                                       \frac{\epsilon}{1-\bar \alpha}
                                       \left(1 + \delta \frac{K_\epsilon}{1-\gamma_\epsilon} \right)
                                       + \epsilon \delta  (\mu V)\sum_{k=1}^{n} \bar \alpha^{n-k} \gamma_\epsilon^k+
                                       \bar \alpha^n  d_\beta(\mu , \nu)
  \\
&\leq
                                       \frac{\epsilon}{1-\bar \alpha}
                                       \left(1 + \delta \frac{K_\epsilon}{1-\gamma_\epsilon}\right)
                                       + \epsilon \delta ( \mu V
  ) (\bar \alpha \vee \gamma_\epsilon)^{n-1}n +
                                       \bar \alpha^n  d_\beta(\mu , \nu). \label{eq:dPnPen}
\ee
This result can be compared to \cite[Theorem 3.1]{rudolf2015perturbation},
which uses a condition similar to Assumption \ref{ass:ClosenessTV}. 
We note that Assumption \ref{ass:lyapunovPeps} is implied by Assumption 
\ref{ass:closeness} when $\epsilon\delta < 1-\gamma$; a simple argument is 
given in the proof of Remark \ref{rem:Harris}. Also under Assumption 
\ref{ass:lyapunovPeps}, if $\mu_0$ and $\mu_\epsilon$ are invariant measures 
of $\P$ and $\P_\epsilon$ respectively then we have the bound
\begin{align*}
  d_\beta(\mu_0, \mu_\epsilon) \leq    \frac{\epsilon}{1-\bar \alpha} \left(1 + \delta \frac{K_\epsilon}{1-\gamma_\epsilon} \right).
\end{align*}

Notably, we use no special features of the weighted total variation metric in proving
any of the above results, or, more formally
\begin{remark}
 All of the previous results hold with $d_\beta$ replaced by any metric in which $\P$ is 
 a strict contraction when Assumption \ref{ass:closeness} holds in the same
 metric. The key ingredients are therefore strict contraction in a metric $d$ and Assumption 
 \ref{ass:closeness} in the same metric $d$.
\end{remark}

We now show that under the following additional condition, one can prove Harris' theorem for $\P_\epsilon$. 
\begin{assumption} \label{ass:ClosenessTV}
For every $0<R< \frac{2(K+ \epsilon)}{1-(\gamma + \epsilon \delta)}$ there 
exists $\zeta 
< 1-\bar \alpha_0$ (depending on $R$) such that 
$\sup_{x \in \mathcal S(R)} d_0(\delta_x \P_\epsilon,\delta_x \P) \le \zeta$  
for $\mathcal S(R) = \{x : V(x) \le R\}$. 
\end{assumption}
This result is included mainly 
for completeness. It is common in the MCMC literature to prove Harris' theorem, 
and many practitioners mistakenly interpret it as a guarantee of good 
finite-time 
performance. It is clear from Theorem \ref{thm:ErrorBound} that this 
is not necessary to obtain the kind of variation bounds that are desired in 
MCMC applications, but the following result may nonetheless be of interest.

\begin{remark} \label{rem:Harris}
 Suppose Assumptions \ref{ass:small}, \ref{ass:closeness}, 
\ref{ass:lyapunov}, and \ref{ass:ClosenessTV} hold and $\delta 
\epsilon<1-\gamma$. Then there exists $\bar \alpha_\epsilon < 1$ and $\beta > 
0$ such that 
$\rho_\beta(\P_\epsilon \mu, \P_\epsilon \nu) \le \bar{\alpha_\epsilon} 
\rho_\beta(\mu,\nu)$
for any probability measures $\mu, \nu$ on $\mathbf X$. 
\end{remark}
\begin{proof}
We have 
\be
\P_\epsilon V &= (\P + \P_\epsilon - \P) V \le \gamma V + K + \epsilon(1+\delta V),
\ee
and for any $x,y \in \mathcal S(R)$ with $\mathcal S(R) = \{x : V(x) \le R\}$
\be
d_0(\delta_x \P_\epsilon, \delta_y \P_\epsilon) &\le d_0( \delta_x \P_\epsilon,\delta_x \P) + d_0(\delta_x \P, 
\delta_y \P) + d_0(\delta_y \P, \delta_y \P_\epsilon) \\
&\le \zeta + 2\bar \alpha_0 + \zeta = 2(\bar \alpha_0+\zeta).   
\ee
for every $R \le \frac{2(K+\epsilon)}{1-(\gamma+\epsilon \delta)}$.
\end{proof}

Our main variation bound is given by the following result. 
\begin{theorem} \label{thm:ErrorBound}
Assume that Assumptions \ref{ass:lyapunov}, \ref{ass:small}, \ref{ass:closeness}
and \ref{ass:lyapunovPeps} hold. Then there exists $C,c_0,c_1 < \infty$ so that for any
$|\varphi| < \sqrt{V}$
\begin{multline*}
  \mathbf E \left( \frac1n \sum_{k=0}^{n-1} \varphi(X_k^\epsilon) - \mu_0 \varphi  \right)^2\le 3 C^2 \epsilon c_0 + \frac{3 C^2}{n} \left(2+ \frac{2K_{\epsilon}}{1-\gamma_\epsilon} + \frac{ \epsilon \delta c_1 V(x_0)}{1-\sqrt{\gamma_\epsilon}} \right) + \mathcal{O}\left( \frac{1}{n^2} \right),
\end{multline*}
with $X_0^\epsilon = x_0$ and $X_k^\epsilon \sim \delta_{x_0} \P_\epsilon^{k-1}$. Moreover, the constants $C,c_0,c_1$ satisfy
\be
C &\le \frac{1 \wedge \mu_0 V}{1-\bar \alpha_{(1/2)}}, \quad c_0 \le 2 + 5 \frac{(\delta \vee \sqrt{\delta})(K_\epsilon \vee \sqrt{K_\epsilon})}{(1-\sqrt{\gamma_\epsilon})^2} \quad c_1 = \left( 2 + \frac{\sqrt{K_\epsilon}}{1-\sqrt{\gamma_\epsilon}}  \right),
\ee
where $\mu_0$ is the unique invariant measure of $\P$ and $1-\bar \alpha_{(1/2)}$ is the spectral gap in the weighted total variation norm built on $V^{1/2}$ with an appropriate $\beta_{(1/2)}>\beta$.
\end{theorem}
The proof is deferred to the Appendix. The bound consists of an error term 
that goes to zero when $\epsilon \to 0$, terms that goe to 
zero at the rate $n^{-1}$, and terms going to zero like $n^{-2}$. This result
is also quite general.

\begin{remark}
 The proof of Theorem \ref{thm:ErrorBound} required that $\P$ is a strict contraction in $d_\beta$ only to prove that the potential $U = \sum_{k=0}^{\infty} \P^k(\varphi - \mu \varphi)$ is well-defined and has bounded $\lvertiii \cdot \rvertiii_\beta$ norm for $|\varphi| < \sqrt{V}$. The remaining parts of the argument require Assumption \ref{ass:closeness} and the Lyapunov condition. It follows that the result holds with $d_\beta$ replaced by a general lower semicontinuous metric $d$ for which $\P$ is a strict contraction in the associated dual metric. The result is then true for functions $\varphi$ for which $\varphi^2$ has bounded Lipschitz-$d$ norm. 
\end{remark}

\section{Motivating applications}
We give an overview of a few statistical applications for this work. The 
results in the previous section are well-suited to some of these applications, but
not others. This will motivate additional results in Wasserstein metrics that
appear in the following section.

\subsection{Bayesian posterior measures}
The target measure in Bayesian statistics is the posterior distribution of 
parameters given data. One obtains the posterior by first specifying a sampling model
$L(z,x)$, the distribution of observables/data $z$ given an unknown state of
nature $x$. Here $L(z,x)$ is taken to be a density with respect to Lebesgue or counting measure
for each $x \in \X$, with $\X$ the parameter space (often $\R^p$). The model is completed by
a prior $\pi$ that expresses the statistician's beliefs about the state of nature $x$
before observing $z$. We take $\pi$ to be a density with respect to a dominating measure on $\X$, 
usually Lebesgue measure on $\R^p$. The posterior measure $\mu$ expresses the statistician's updated
beliefs about $x$ after observing $z$, and has density $m$ satisfying
\be
m(x) \propto L(z,x) \pi(x),
\ee
where $\propto$ indicates that $m$ involves some unknown constants not 
depending on $x$. Bayesian statisticians seek to compute expectations with respect to $\mu$,
and often do so by constructing a Markov kernel $\P$ with invariant measure $\mu$, then
using pathwise time averages to approximate expectations with respect to $\mu$.

\subsection{Bayesian generalized linear models}
In Bayesian generalized linear models, $L$ takes the form
\be \label{eq:GLMlik}
\log L(x,z) \propto \sum_{i=1}^N g^{-1}(\langle w_i,x \rangle) z_i - A(\langle w_i,x \rangle)
\ee
where $w_i$ for $i=1,\ldots,N$ 
are predictors/covariates with each $w_i \in \R^p$, $z_i \in \R$, and $g^{-1}, A$ are
real-valued functions. A common prior choice is the multivariate Gaussian with 
$\log \pi(x) \propto \|B^{-1/2} (x - b)\|^2/2$, where $B$ is a positive-definite matrix and $b \in \R^p$.
Two popular approaches to approximating expectations with respect to this distribution are data augmentation
Gibbs sampling and random walk Metropolis. We consider approximation schemes for each.

Data augmentation Gibbs samplers introduce additional variables $\omega$ for which
\be
L(z,x) = \int L^*(\omega,z) f(\omega,x) d\omega,
\ee
where $f(\omega,x)$ is the joint density of $\omega,x$. 
A Gibbs sampling algorithm with update rule that iterates sampling from the conditional distributions
of
\be
x &\mid \omega, b,B \\
\omega &\mid x,z
\ee
has $x$-marginal invariant measure $\mu$, the posterior distribution arising 
from the original sampling model $L$ and
prior $\pi$. The advantage of this strategy is that often $f(\omega,x)$ can be chosen
such that both of the above conditionals are known distributions that are easy to sample from. Often, 
the conditional distribution of $x \mid \omega, b, B$ is $p$-variate Gaussian.

It is common that the conditional distribution of $x \mid \omega, b, B$ depends on $\omega$ only through
sums of $\omega$, or higher-order sample moments, with some entries of $\omega$ being independent and 
identically distributed given $x,z$. In large samples (i.e. $N$ is large in \eqref{eq:GLMlik}), one can easily need to sample millions
of latent variables $\omega$ at each iteration, only to condition on sums over large groups of these variables 
when updating $x$. In this case, it is natural to forego direct sampling of the $\omega$ and just sample a Gaussian
approximation to the relevant sums, defining an approximate kernel $\P_\epsilon$. 
We consider an approximation of this sort in Section \ref{sec:Probit}. The conditions of 
Theorem \ref{thm:ErrorBound} are natural for this application because: (1) both $\P$ and $\P_\epsilon$ are absolutely
continuous with respect to Lebesgue measure; (2) it is possible to obtain bounds on total variation error for
such approximations; and (3) many data augmentation Gibbs samplers satisfy Assumptions \ref{ass:lyapunov} and 
\ref{ass:small}. 

An alternative to data augmentation is to use a random-walk Metropolis algorithm to target $\mu$. The requires computing
$m(y)/m(x)$ at different points $y,x$ in the state space at each iteration. Typically this has computational cost that is
at least linear in the number of data $N$. Minibatching Metropolis reduces computation time by replacing $\log L(z,x)$ with
\be \label{eq:MinibatchTarget}
\log L_{\mathcal A}(z,x) = \frac{N}{|\mathcal A|} \sum_{i \in \mathcal A} g^{-1}(\langle w_i,x \rangle) z_i - A(\langle w_i,x \rangle)
\ee
for a (usually random) $\mathcal A \subset \{1,2,\ldots,N\}$. The subset $\mathcal A$ may be fixed or it may change at each
iteration. We analyze minibatching Metropolis in Section \ref{sec:MHMinibatch}, where we consider a version of the algorithm
that resamples random subsets $\mc A$ of a pre-specified size $|\mc A| = N_0(\epsilon)$ at each iteration. Clearly the smaller 
the value of $\epsilon$, the larger the value of $N_0(\epsilon)$ necessary to achieve the desired approximation error.
This application is also well-suited to the conditions
in Section \ref{sec:Bounds} since: (1) both $\delta_x \P$ and $\delta_x \P_\epsilon$ are absolutely continuous with respect to Lebesgue
measure for each $x \in \X$; (2) in most cases, both satisfy Harris' theorem and have the same Lyapunov function; (3) it is possible
to show approximation error results in total variation (though, as we shall see, small approximation error often means using most of the
data, negating the computational advantages). 

\subsection{Bayesian Spatial and Nonparametric Modeling} \label{sec:GPSummary}
Another class of statistical applications in which approximations are often used is spatial modeling with Gaussian processes. 
In this application, data are noisy observations of a function $f : \mathbf W \to \R$ 
for some index set $\mathbf W$. For example, $f(w)$ might be the temperature at location $w$ on the earth's surface $\mathbf W$.
For a vector $z = (z_1,\ldots,z_N)$ of observables at locations $w_i,\, i=1,\ldots,N$, the sampling model is
\be \label{eq:GPLikelihood}
\log L(z,x) \propto -\frac12 \log |x_3^2 \{I + x_2^2 \Sigma\}| - \frac12 \| \{x_3^2 I + x_3^2 x_2^2 \Sigma\}^{-1/2} z\|^2
\ee
where $\Sigma$ is a $N \times N$ positive-defininite matrix-valued function of $x_1$ and $W$ with entries
\be
\Sigma_{ij} = \phi(x_1, \rho(w_i,w_j)),
\ee
where $\phi$ is a positive-definite function and $\rho: \mathbf W \times \mathbf W \to \R_+$ is a metric on $\mathbf W$.
A simple example is when $\phi$ is a Gaussian kernel and $\rho$ is the squared Euclidean distance, for which 
$\log(\Sigma_{ij}) = -x_1^2 \|w_i-w_j\|^2$. 

The parameter space $\X$ is only three dimensional, so it might seem natural to construct a Metropolis-Hastings algorithm
to target $\mu$ after choosing some appropriate $\pi$. While this approach is common, it suffers from a serious limitation. 
Computing the ratios $m(y)/m(x)$ requires solving a linear system in the matrix $\Sigma$, and because $\Sigma$ is a function
of a state variable $x_1$, the system must be solved at each iteration. To make computation feasible when $N$ is large, 
practitioners typically make two approximations: (1) they discretize the proposal kernel in the first dimension $\X_1$, so that
the Markov chain will revisit the same values of $x_1$ multiple times and the linear system solutions can be re-used; and (2)
they utilize low-rank approximations to $\Sigma$ or truncated Karhunen-Loeve expansions to approximate the sampling model,
allowing linear systems to be solved more quickly in the lower-dimensional space. 
This application is \emph{not} a good fit for the results in Section \ref{sec:Bounds} because $\P$ and $\P_\epsilon$ are mutually
singular in the sense of measure. It is more natural to study these approximations in Wasserstein-1 metrics and Wasserstein-like semimetrics.

\section{Results in Wasserstein Metrics}
In addition to making natural the study of discrete approximations of continuous measures, the Wasserstein metric
can also be more attractive practically, as total variation is often too strong of a metric by which to assess approximation
accuracy. This can give the misleading impression that a proposed approximation is useless, while in reality it gives small
approximation error in a Wasserstein norm. As such, a theory of approximating Markov chains in the Wasserstein metric is quite attractive,
particularly in high-dimensional settings.

\subsection{Approximations in Unweighted Wasserstein}
We first give an approximation error result in the case where $\P$
is a strict contraction in an unweighted Wasserstein metric.  

Henceforth we will consider the distance  
\be \label{eq:ddef}
d(x,y) = 1 \wedge \frac{|x-y|}{\delta},
\ee
which generates the same topology as the standard distance but is localized
on a scale $\delta$ and capped at one. We define a 
Lipschitz seminorm on measurable functions $\lvertiii \varphi \rvertiii$ as in \eqref{eq:LipschitzNorm}, 
and the associated dual metric on probability measures $d(\mu,\nu)$ as in \eqref{eq:DualMetric}, but
with the distance in \eqref{eq:ddef} substituted for $d_\beta$. 
We write $\Lip_c(d)$ to denote the set of functions with Lipschitz-$d$ norm less than
$c$. Notice that because $d$ is capped
at one, if $\varphi \in \Lip_c(d)$ for $c < \infty$, then $\varphi$ is necessarily bounded.

We now give conditions sufficient to show a strict contraction in this metric.
The first condition on $\P$ states that $\P$ is locally Lipschitz in the
initial condition:
\begin{assumption} \label{ass:LipschitzKernel}
There exists $C < \infty$ such that for $|x-y| < \delta$
\be
d(\delta_x \P,\delta_y \P) < C |x-y|.
\ee
\end{assumption}

Our second condition is a form of
uniform
topological irreducibility. 
\begin{assumption} \label{ass:CouplingBall}
For all $\gamma >0$ and $(x,y) \in \X \times \X$ there exists $\Gamma_{x,y} \in 
\mathcal C(\delta_x \P,\delta_y \P)$ 
and $\alpha_\gamma > 0$ such that $\Gamma_{x,y}( (a,b) : |a-b|< \gamma ) > 
\alpha_\gamma$. 
\end{assumption}

Under these assumptions, we have the following contractility result
which implies exponential convergence in the Wasserstein metric.
\begin{theorem} \label{thm:UniformWasserstein}
Suppose Assumptions \ref{ass:LipschitzKernel} and \ref{ass:CouplingBall} hold. 
Then there exists $\bar \alpha < 1$ such that 
\be
d(\delta_x \P, \delta_y \P) \le \bar \alpha d(x,y).
\ee
\end{theorem}
\begin{proof} This proof largely follows Section~2.1 from \citet{HM_06}. 
First suppose $|x-y| < \delta$ and $\gamma < \delta < \frac1C$. Then
\be
d(\delta_x \P,\delta_y \P) &\le C |x-y| 
\le C \delta \bigl(1 \wedge \tfrac{|x-y|}{\delta} \bigr) \le C \delta d(x,y).
\ee
On the other hand if $|x-y|> \delta$ then defining $\Delta_\gamma = \{ (a,b) \in 
\X \times \X : |a-b| < \gamma \}$
\be
d(\delta_x \P, \delta_y \P) &\le \int_{\Delta_\gamma} d(a,b) \Gamma_{x,y}(da,db) 
+ \int_{\Delta_\gamma^c} d(a,b) \Gamma_{x,y}(da,db) \\
&\le \frac{\gamma}{\delta} \alpha_\gamma + (1-\alpha_\gamma) = 
1-\left(1-\frac{\gamma}{\delta} \right) \alpha_\gamma \le \bar \alpha_{\gamma} = \bar \alpha_{\gamma} d(x,y).
\ee
Putting $\bar \alpha = \bar \alpha_\gamma \wedge C\delta$ completes the proof.
\end{proof}

So far we have said nothing of approximations. We will proceed under the following special case
of Assumption \ref{ass:ClosenessTV} in the Wasserstein metric, which arises from taking
$V(x)$ constant.
\begin{assumption} \label{ass:WassersteinApprox}
$\P, \P_\epsilon$ satisfy
 $\sup_{x \in \X} d(\delta_x P, \delta_x \P_\epsilon) < \epsilon$.
\end{assumption}

This gives immediately a Corollary of Theorem \ref{thm:TVperturb}
\begin{corollary}
 Suppose Assumptions \ref{ass:WassersteinApprox}, \ref{ass:LipschitzKernel}, and \ref{ass:CouplingBall} 
 hold. Then
 \be
 d(\delta_x \P, \delta_y \P_\epsilon) \le \bar \alpha d(x,y) + \epsilon.
 \ee
\end{corollary}

Analogues of the various corollaries of Theorem \ref{thm:TVperturb} in the Wasserstein metric then follow immediately.
To exhibit a different approach to proving such results, we give the following analogue of \eqref{eq:dPnPen} for time
averages. The proof is similar to the proof of Theorem \ref{thm:ErrorBound} in the use of the Poisson
equation.
\begin{theorem} \label{thm:WassersteinTimeAvgBound}
 Suppose Assumptions \ref{ass:WassersteinApprox}, \ref{ass:LipschitzKernel}, and \ref{ass:CouplingBall} 
 hold. Put $\widetilde \varphi = \varphi - \mu \varphi$. Then
 \be
 \mathbf E \left[\frac1n \sum_{k=0}^{n-1} \widetilde \varphi(X_k^\epsilon)\right] &\le \frac1n \frac2{1-\bar \alpha} + \frac{\epsilon}{1-\bar \alpha}.
 \ee
\end{theorem}
\begin{proof}
As in the proof of Theorem \ref{thm:ErrorBound}, we work initially with the generator and Poisson equation. Define
$L = \P-I$ and $U(x) = \sum_{k=0}^{\infty} \P^k \widetilde \varphi(x)$,
with $\phi \in \Lip_1(d)$, so that $LU = -\widetilde \varphi$. Then, with the same notation as in \eqref{eq:PoissonEquation},
\be
\frac1n \sum_{k=0}^{n-1} \widetilde \varphi(X_k^\epsilon) = \frac{U(X_0^\epsilon) - U(X_n^\epsilon)}n + \frac1n M_n^\epsilon + \frac1n \sum_{k=0}^{n-1} (\P_\epsilon - \P)U(X_k^\epsilon), \label{eq:PoisDecomp}
\ee
where $M_n^\epsilon$ is a Martingale. Observe that
\be
U(x) - U(y) = \sum_{k=0}^{\infty} \P^k \varphi(x) - \P^k \varphi(y) \le \sum_{k=0}^{\infty} \bar \alpha^k d(x,y) = \frac1{1-\bar \alpha} d(x,y),
\ee
so that if $\varphi \in \Lip_1(d)$, then $U \in \Lip_{\frac1{1-\bar \alpha}}(d)$. So applying Assumption \ref{ass:WassersteinApprox},
\be
\mathbf E \frac1n \sum_{k=0}^{n-1} \widetilde \varphi(X_k^\epsilon) &= \frac{\mathbf E[U(X_0^\epsilon) - U(X_n^\epsilon)]}n + \frac1n \sum_{k=0}^{n-1} \mathbf E (\P_\epsilon - \P)U(X_k^\epsilon) \\
&\le \frac1n \frac2{1-\bar \alpha} + \frac1n n \frac{\epsilon}{1-\bar \alpha} = \frac1n \frac2{1-\bar \alpha} + \frac{\epsilon}{1-\bar \alpha},
\ee
completing the proof.
\end{proof}
Clearly, the difference relative to the total variation case is in proving a strict contraction for $\P$, for 
which the conditions differ. Conditions \ref{ass:CouplingBall} and \ref{ass:LipschitzKernel} can be compared 
directly to Doeblin's condition $\sup_{x,y} d_0(\delta_x \P,\delta_y \P_\epsilon)$, with $d_0$ the ordinary
total variation metric, which gives a strict contraction in total variation. In contrast, in the Wasserstein
metric we need two conditions: uniform topological irreducibility and a (local) d-Lipschitz condition on the 
exact kernel $\P$. This is similar to the weighted case in the next section, for which we require three conditions on $\P$,
in contrast to the two conditions needed for Theorem \ref{thm:HM2011}.

\subsection{$V$-weighted Wasserstein-like semimetrics} \label{sec:WeightedWasserstein}
We now consider a non-metric notion of convergence more appropriate for the unbounded state space setting commonly encountered in applications. 
The results in this section are alternatives to the results in Section \ref{sec:Bounds} that are better suited to 
the case where $\P_\epsilon$ involves discretization, or where the total variation metric is too strong to be useful.
For example, if $\P_\epsilon$ utilizes Gaussian approximations to sums of discrete random variables, one often wants to 
use the central limit theorem and Berry-Esseen bounds to obtain an approximation error result. These results will hold
in Wasserstein metrics, but not total variation, making the following results more appropriate to studying these types
of approximations. We note that some additional effort is needed to actually obtain a CLT using the machinery in this section,
see e.g. \citet[Section 4.1.1]{hairer2014spectral}.

Our development builds on the weak Harris theorem of \cite{hairer2011asymptotic}. 
Suppose that $V$ is a continuous Lyapunov function of $\P$. A function $d: \X \times \X \to \R_+$ is said to
be ``distance-like'' if it is symmetric, lower semicontinuous, and satisfies $d(x,y) = 0 \Leftrightarrow x = y$; that it,
$d$ is a lower semicontinuous semimetric.
Here we will assume that $d: \X \times \X \to [0,1]$; see \cite[Remark 4.7]{hairer2011asymptotic} for a discussion
of why this restriction is innocuous.
For a distance-like function $d$, define a positive function 
on probability measures $d(\mu,\nu)$ by 
\be \label{eq:CouplingDistanceLike}
d(\mu,\nu) = \inf_{\Gamma \in \mc C(\mu,\nu)} \int_{\X \times \X} d(x,y) \Gamma(dx,dy).
\ee
The function $d$ will play the same role as the metric $1 \wedge \delta^{-1} |x-y|$ in the unweighted Wasserstein case,
or $2 \1\{ x \ne y\}$ in the unweighted total variation case. Now, 
for $\beta> 0$, define a weighted form of $d(x,y)$ by
\be 
\tilde d_\beta(x,y) = \sqrt{d(x,y) (2 + \beta V(x) + \beta V(y))},
\ee
with $\tilde d_\beta(\mu,\nu)$ as in \eqref{eq:CouplingDistanceLike}. Since $V$ is continuous, $\tilde d_\beta$
is also distance-like, and is in some sense analogous to the $V$-weighted metric $d_\beta$ in \eqref{eq:dBeta}.

We now give conditions that are sufficient to ensure that $\P$ has a ``spectral gap'' 
in $\tilde d_\beta$ for some $\beta > 0$ (in the sense of Theorem \ref{thm:WtdWasserstein}). 
We assume that $\P$ satisfies a contraction condition with respect to the distance-like function $d$.
\begin{assumption} \label{ass:contracting}
The distance-like function $d : \X \times \X \to [0,1]$ is contracting for $\P$ if there exists $\bar \alpha_0 \in (0,1)$ such that
$d(\delta_x \P, \delta_y \P) \le \bar \alpha_0 d(x,y)$, 
for every $x,y$ such that $d(x,y) < 1$
\end{assumption} 
and that sublevel sets of $V$ satisfy a condition similar to the minorization condition used in Section \ref{sec:Bounds}
\begin{assumption} \label{ass:dmsall}
For every $R>0$, sublevel sets $\mathcal S(R) = \{ x: V(x) < R\}$ of $V$ satisfy
\be
\sup_{x,y \in \mathcal S(R)} d(\delta_x \P,\delta_y \P) < (1-\alpha), 
\ee
for an $\alpha \in (0,1)$ that depends on $R$.
\end{assumption}

Under Assumptions \ref{ass:contracting} and \ref{ass:dmsall}, we have the following result. This result is essentially a discrete-time version of the ``weak Harris'' theorem \cite[Theorem 4.8]{hairer2011asymptotic}, with the constant $\beta$ tuned such that $\tilde d_\beta$ contracts in one step, resulting in an analogue of Theorem \ref{thm:HM2011}. The argument given here is more natural in the discrete time setting and has many parallels with the discrete time proof of the Harris theorem in \cite{hairer2011yet}, emphasizing the similarity of the ``weak'' and ordinary Harris theorem. Enough of the details differ from \cite{hairer2011asymptotic} that we give the entire proof.
\begin{theorem} \label{thm:WtdWasserstein}. 
Suppose Assumptions \ref{ass:contracting} and \ref{ass:dmsall} hold for a distance-like function $d : \X \times \X \to [0,1]$ and $V$ is a continuous Lyapunov function of $\P$. Then there exists $\bar \alpha \in (0,1)$ and $\beta > 0$ such that
$\tilde d_\beta(\delta_x P, \delta_y P) \le \bar \alpha \tilde d_\beta(x,y)$. 
\end{theorem}
\begin{proof}
First consider the case $d(x,y) = 1$ and $V(x) + V(y) \ge R$ for $R>0$. 
In this case
\be
\tilde d_\beta^2(x,y) &= 2 + \beta (V(x) + V(y)) 
\ee
Choosing $\bar \gamma \in (\gamma,1)$ and using the Lyapunov structure, we have
\be
\tilde d_{\beta}(\delta_x \P, \delta_y \P)^2 &= \left( \inf_{\Gamma \in \mathcal C(\delta_x \P, \delta_y \P)} \int \tilde d_\beta(u,v) \Gamma(du,dv) \right)^2 \\
&\le \left( \int \tilde d_\beta(u,v) \Gamma^*(du,dv) \right)^2 \le \int \tilde d^2_\beta(u,v) \Gamma^*(du,dv) \\
&= \int (2 + \beta V(u) + \beta V(v))  \Gamma^*(du,dv) \\
&\le 2 + \beta \gamma (V(x) + V(y)) + 2 \beta K \\
&\le 2 + \beta \bar{\gamma} (V(x) + V(y)) + \beta (2K - (\bar \gamma - \gamma) R),
\ee
where $\Gamma^*$ is any coupling. 
So if $R$ is large enough that $(\bar \gamma - \gamma) R > 2K$, then it follows
$\tilde d_{\beta}(\delta_x \P, \delta_y \P)^2 \le \bar \alpha_1 \tilde d_\beta^2(x,y)$ 
for $\bar \alpha_1 < 1$. Clearly this choice of $R$ implies that
 $R>\frac{2K}{1-\gamma}$, a familiar quantity from
the total variation case. Up to this point, $\beta$ could be any positive number, with only
the exact value of $\bar \alpha_1$ (and not the existence of $\bar \alpha_1 < 1$) depending 
on the value of $\beta$. The remaining two parts determine the value of $\beta$.

Next, suppose $d(x,y) < 1$ and let $\Gamma_{x,y} \in \mc C(\delta_x \P, \delta_y \P)$. Then by Cauchy-Schwartz
\be
\tilde d_\beta(\delta_x P, \delta_y P)^2 &\le \inf_{\Gamma_{x,y}} \int d(u,v) \Gamma_{x,y}(du,dv) \int (2 + \beta V(u) + \beta V(v)) \Gamma_{x,y}(du,dv) \\
&\le \bar \alpha_0 d(x,y) (2 + \beta \gamma (V(x) + V(y)) + 2 \beta K). 
\ee
Now for $\bar \alpha_2 \in (\bar \alpha_0,1)$ we can choose $\beta$ sufficiently small that $\bar \alpha_0 (2\beta K + 2) \le 2 \bar \alpha_2$, giving
$\tilde d_\beta(\delta_x P, \delta_y P)^2 \le \bar \alpha_2 \tilde d^2_\beta(x,y)$. 
For example, if we take $\bar \alpha_2 = \frac{\bar \alpha_0+1}2$, we need $\beta \le \frac1{2K} \left( \frac{1}{\bar \alpha_0} -1 \right)$. 

Finally consider the case where $d(x,y)=1$ and $V(x)+V(y) \le R$, so that $\tilde d_\beta^2(x,y) > 2$. We use the minorization condition on sublevel sets to conclude that there exists a coupling $\Gamma_{x,y} \in \mathcal C(\delta_x P, \delta_y P)$ for which
\be
\tilde{d}_\beta(\delta_x \P, \delta_y \P)^2 &\le \int d(u,v) \Gamma_{x,y}(du,dv) \int (2 + \beta V(u) + \beta V(v)) \Gamma_{x,y}(du,dv) \\
&\le (1-\alpha) (2 + 2 \beta K + 2 \beta \gamma (V(x) + V(y))) \\
&\le 2-2\alpha  +  2 \beta (K +  \gamma R)
\ee
Since $\beta$ is independent of $\alpha$, we can, for example, take $\beta = \frac{\alpha}{2(K+\gamma R)}$ so that 
\be
\tilde{d}_\beta(\delta_x \P, \delta_y \P)^2 \le 2 - 2 \alpha + \alpha = 2-\alpha < (1-\alpha/2) \tilde d^2_\beta(x,y),
\ee
since $\tilde d^2_\beta(x,y) > 2$. So taking
$\beta = \frac{\alpha}{2(K+\gamma R)} \wedge \frac{1}{2K} \left( \frac{1}{\bar \alpha_0} - 1 \right)$.
concludes the proof.
\end{proof}
The following is a consequence of the convexity of $\tilde d_\beta$.
\begin{corollary}[\citet{hairer2011asymptotic}]
 Suppose Theorem \ref{thm:WtdWasserstein} holds. Then we have for any two probability measures $\nu_1,\nu_2$ on $\X$
 \be
 \tilde d_\beta(\nu_1 \P, \nu_2 \P) \le \bar \alpha \tilde d_\beta(\nu_1,\nu_2)
 \ee
\end{corollary}
Controlling the approximation error in $\tilde d_\beta$ is slightly complicated by the fact that it is not a metric on $\X$; specifically, it does not satisfy the triangle inequality.  However, the following weaker assumption is enough to obtain the desired results.
\begin{assumption} \label{ass:WeakTriangle}
There exists $C < \infty$ such that
 \be \label{eq:WeakTriangle}
 \tilde d_\beta(x,y) \le C(\tilde d_\beta(x,z) + \tilde d_\beta(z,y)).
 \ee
 for all $x,y,z \in \X$.
 \end{assumption}

We recall the following result from \cite{hairer2011asymptotic}, which guarantees that $\tilde d_\beta$ satisfies Assumption \ref{ass:WeakTriangle} under mild conditions on $V$ when $d(x,y) = 1 \wedge |x-y|$, which is the choice we use later in applications.

\begin{remark}
 Suppose $\X$ is a Banach space, $d(x,y)=1 \wedge |x-y|$, and $V$ grows at most exponentially with $|x|$. Then there exists $0 < C < \infty$ such that Assumption \ref{ass:WeakTriangle} holds. 
\end{remark}

Using this we obtain the following analogue of Theorem \ref{thm:TVperturb}. This result effectively
shows that if Assumption \ref{ass:WeakTriangle} holds and we have a one-step approximation
error bound in $\tilde d_\beta$, then we can bound closeness of invariant measures and 
obtain an approximation error result like Theorem \ref{thm:TVperturb} on a time scale that
is lengthened by $(-\log C) (\log \bar \alpha)^{-1}$. 

\begin{theorem} \label{thm:WtdWErrorBound}
Suppose there exists a function $c(n) : \mathbb N \to [0,\infty)$ and a Lyapunov function $\tilde V \le V + \frac{K_\epsilon}{1-\gamma_\epsilon}$ of $\P$ and $\P_\epsilon$ such that for each $n > 0$
\be \label{eq:nStepError}
\sup_x \tilde d_\beta(\delta_x \P^n,\delta_x \P^n_\epsilon) \le \epsilon c(n) (1+\delta \tilde V(x)). 
\ee 
Then for any $\bar \alpha^* \in (0,1)$ there exists $n > (-\log C) (\log \bar \alpha)^{-1}$ such that
\be
\tilde d_\beta (\nu \P^n, \mu \P^n_\epsilon ) \le \bar \alpha^* \tilde d_\beta(\mu,\nu) + \epsilon C c(n) (1+\delta \tilde V(x))
\ee
and if $\mu,\mu_\epsilon$ are any invariant measure of $\P,\P_\epsilon$, respectively then for any $n$ such that
$\bar \alpha^n < C^{-1}$ we have
\be
\tilde d_\beta (\mu, \mu_\epsilon ) \le \frac{C \epsilon c(n)}{1-C \bar \alpha^n} (1+ \delta \mu_\epsilon \tilde V).
\ee
In particular, if Assumption \ref{ass:closeness} holds for $\tilde d_\beta$ then we can take $c(n) = \sum_{j=1}^n C^j \bar \alpha^{j-1}$, so no additional assumptions are needed when $d$ satisfies \eqref{eq:WeakTriangle}. 
\end{theorem}
\begin{proof}
By \eqref{eq:WeakTriangle} it follows
\be
\tilde d_\beta(\delta_y \P^n, \delta_x \P^n_\epsilon) \le C \{ \bar \alpha^n \tilde d_\beta(x,y) + \epsilon c(n) (1+\delta \tilde V(x)) \},
\ee
so that if $n$ is large enough that $C \bar \alpha^n < 1$, we do indeed have 
\be
\tilde d_\beta(\delta_y \P^n, \delta_x \P^n_\epsilon) \le \bar \alpha^* d_\beta(x,y) + \epsilon C c(n) (1+\delta \tilde V(x)),
\ee
and we can make $\bar \alpha^*$ arbitrarily close to zero by choosing $n$ sufficiently large. Moreover,
for any invariant measures $\mu,\mu_\epsilon$ of $\P,\P_\epsilon$
\be
\tilde d_\beta (\mu,\mu_\epsilon ) \le \frac{C \epsilon c(n)}{1-C \bar \alpha^n} (1+ \delta \mu_\epsilon \tilde V),
\ee 
again for any $n$ such that $C \bar \alpha^n < 1$. 
Further when  Assumption \ref{ass:closeness} holds for a $V$ which is a
Lyapunov function of  $\P_\epsilon$ then \eqref{eq:WeakTriangle} implies
\eqref{eq:nStepError} since 
\be
\tilde d_\beta(\delta_x \P^n, \delta_x \P^n_\epsilon) &\le C \{ \bar \alpha \tilde d_\beta( \delta_x \P^{n-1}, \delta_x \P^{n-1}_\epsilon) + \epsilon (1+\delta \delta_x \P_\epsilon^{n-1} V) \} \\
&\le \epsilon \sum_{j=1}^n C^j \bar \alpha^{j-1} (1+\delta \delta_x \P_\epsilon^{n-j} V ) \\
&\le \epsilon \sum_{j=1}^n C^j \bar \alpha^{j-1} \left\{ 1+\delta \left( \frac{1-\gamma_\epsilon^{n-j}}{1-\gamma_\epsilon} K_\epsilon + \gamma_\epsilon^{n-j} V(x) \right)  \right\} \\
&\le \epsilon \sum_{j=1}^n C^j \bar \alpha^{j-1} \left\{ 1+\delta \left( \frac{K_\epsilon}{1-\gamma_\epsilon} + V(x) \right)  \right\} \\
&\equiv \epsilon c(n) (1+\delta \tilde V(x))
\ee
\end{proof}

\section{Applications}
We begin this section with a result which is helpful in verifying the Assumption \ref{ass:closeness}.
We then
apply the previous sections to 
data augmentation Gibbs sampling, Minibatching Metropolis-Hastings, and 
approximate MCMC for Gaussian process models.

\subsection{Achieving the error condition} \label{sec:Achieving}
For the results in Section \ref{sec:Bounds} to be practical, we need a way to construct approximations
$\P_\epsilon$ that achieve Assumption \ref{ass:closeness} that is broadly
applicable. Often, it is easier to construct approximations satisfying a
condition like
$d_0(\delta_x \P, \delta_x \P_\epsilon) < \epsilon$
than to directly construct an approximating kernel with error depending 
on the Lyapunov function. However, uniform total variation error control
is not enough to show Assumption \ref{ass:closeness}, so we seek an adaptive
total variation error condition that gives Assumption \ref{ass:closeness}. 

Suppose $V$ is a Lyapunov function of both $\P$ and $\P_\epsilon$. Observe that 
for every $\epsilon > 0$ there exists $M_\epsilon(x)<\infty$ such that
\be 
\sup_{|\varphi|<V} \int \varphi(y) \mathbf 1\{|\varphi(y)| > M_\epsilon(x)\} \P(x,dy) < \epsilon/4. \label{eq:BoundedTail}
\ee
A similar condition holds for $\P_\epsilon(x,dy)$ for each $x$; redefine $M_{\epsilon}(x)$ so that the condition in \eqref{eq:BoundedTail} holds for both $\P$ and $\P_\epsilon$. Now suppose 
\be \label{eq:TVBoundMx}
d_0(\delta_x \P, \delta_x \P_\epsilon) < \frac{\epsilon \gamma V(x)}{2 M_{\epsilon}(x)} + \frac{\epsilon}{4 M_\epsilon(x)}.
\ee
Then setting $\mathcal{A}_\epsilon=\{|\varphi(y)| > M_\epsilon(x)\}$ for any $|\varphi| < V$
\be
\int \varphi(y) (\P-\P_\epsilon)(x,dy) &= 
\int \varphi(y) (\mathbf 1_{\mathcal A_\epsilon}(y) + \mathbf 1_{\mathcal A_\epsilon^c}(y)) (\P-\P_\epsilon)(x,dy) \\
&\le \epsilon + \gamma \epsilon V(x) = \epsilon(1+\gamma V(x)).
\ee
In other words, total variation control is good enough, assuming we tune the approximation error in total variation to the current state of the chain. 

 \subsection{Application to Gibbs sampling} \label{sec:Probit}
In this section we consider approximating a Gibbs sampler for binomial probit regression with Gaussian priors. 
This is a generalized linear model, and the log of the target density is given by 
\be \label{eq:Probit}
\log\{ m(x) \} \propto \sum_{i=1}^N z_i \log \frac{\Phi( \langle w_i, x\rangle)}{1-\Phi( \langle w_i, x\rangle)} + m_i \log (1-\Phi( \langle w_i, x\rangle ))- \frac12 \|B^{-1/2}(x-b)\|^2
\ee
where $\Phi$ is the standard Gaussian distribution function. A transition kernel $\P$ with $x$-marginal invariant measure $\mu$ with density $m(x)$ is defined by the data augmentation Gibbs sampler with update rule
\be
\omega_i \mid x &= \sum_{j=1}^{z_i} Z^+_{ij} + \sum_{j=1}^{m_i-z_i} Z^-_{ij} \\
p_{Z^+_{ij}}(z) &\propto \frac{1}{\sqrt{2\pi}} e^{-(z-\langle w_i, x \rangle)/2} \mathbf 1\{ z>0 \}, \quad p_{Z^-_{ij}}(z) \propto \frac{1}{\sqrt{2\pi}} e^{-(z-\langle w_i,x \rangle)/2} \mathbf 1\{ z \le 0 \} \\
x \mid \Omega &\sim \Normal((B^{-1} + M)^{-1} W'\Omega , (B^{-1} + M)^{-1}), \quad M = W'D W.
\ee
Here, $D$ is a diagonal matrix with diagonal entries $m_1,\ldots,m_N$; $W$ is a $N \times p$ matrix with rows consisting of the $w_i$'s; and $\Omega = (\omega_1,\ldots,\omega_N)$ is a $N \times 1$ vector of the $\omega_i$'s. It is common in the literature to analyze $\P$ in the special case of a flat prior on $\beta$, in which case the update is simplified slightly so that the last step becomes
$x \sim \Normal(M^{-1} W'\Omega , M^{-1})$. We proceed to give a result for this case.

An approximating kernel $\P_\epsilon$ of $\P$ can be generated by replacing $\omega_i=U^+_i+U^-_i$ with
\be
U^+_i &\sim \Normal(z_i \mathbf E (Z_{ij}^+), z_i \var(Z_{ij}^+) ), \quad U^-_i &\sim \Normal((m_i-z_i)\mathbf E(Z_{ij}^-), (m_i-z_i) \var (Z_{ij}^-)),
\ee
This changes the sampling cost of the data augmentation step from $\bigO(\sum_i m_i)$ to $\bigO(N)$, a significant savings when the $m_i$ are large. An adaptive approximation can be obtained by using the exact truncated normal sampling when $z_i$ or $m_i-z_i$ is small and using the approximation otherwise. 
We show a Lyapunov function of both $\P$ and $\P_\epsilon$ in the case of a flat prior on $x$. 
\begin{proposition} \label{prop:LyapunovProbit}
 Suppose $0 < \min_i z_i/m_i < 1$. Then $V: \mathbb R^p \to \mathbb R_+$ given by $V(x) = x' W'DW x$ is a Lyapunov function of both $\P$ and $\P_\epsilon$. In particular
 $\P V \le \gamma V + K$ and $\P_\epsilon V \le \gamma V + K$ 
 for some $0< \gamma < 1$ and $K>0$, and we can take $\gamma = 1-(\max_i m_i)^{-1}$.
\end{proposition}
\noindent The proof of Proposition~\ref{prop:LyapunovProbit} is given in the appendix.
Note that the results of \cite{roy2007convergence} can be used directly to verify
that $V(x) = x' W'DW x$ is a Lyapunov function for $\P$ and that $\P$ satisfies a 
minorization condition on sublevel sets of $V$, since the model in \eqref{eq:Probit}
is equivalent to a binary probit model with some rows of $W$ being identical, so the
main interest of this result is to give a simpler proof under slightly stronger conditions,
to give an explicit bound on $\gamma$, and to show that $V$ is also a Lyapunov function
of $\P_\epsilon$.

Now we give a bound on $d_0(\delta_x \P, \delta_x \P_\epsilon)$, which is sufficient
to construct an algorithm satisfying Assumption \ref{ass:ClosenessTV} using the
approach described in Section \ref{sec:Achieving}.

\begin{proposition}
 Suppose $\P$ is the exact data augmentation Gibbs sampler for probit regression, and $\P_\epsilon$ uses Gaussian approximations to $\omega_i$. Then 
 \be
 d_0(\delta_x \P, \delta_x \P_\epsilon) = \frac{\sqrt 2}4 \sqrt{  \sum_{i=1}^N \var(\omega_i \mid x) \psi_{ii} },
 \ee
 with $\psi_ii$ the $i$th diagonal entry of the matrix $W M^{-1} W'$.
\end{proposition}
\begin{proof}
Denote by $\KL(\mu \| \nu)$
the Kullback-Leibler divergence between probability measures $\mu,\nu$ that are
absolutely continuous with respect to a dominating measure $\lambda$, which in
this example we can take to be Lebesgue measure. Denote by
$\mu(x \mid \Omega)$ and $\mu_\epsilon(x \mid \Omega)$ the conditional
measure of $x$ given $\Omega$ induced by the kernels $\P$ and $\P_\epsilon$,
respectively. Observe
\be
\KL(\mu \| \mu_\epsilon) &= \frac12 (M^{-1} W' (\Omega-\Omega_\epsilon))' M  (M^{-1} W' (\Omega-\Omega_\epsilon)) \\
&= \frac12  (\Omega-\Omega_\epsilon)' W M^{-1} W' (\Omega-\Omega_\epsilon).
\ee
Putting $\Psi=W M^{-1} W'$, we have by Pinsker's inequality
\be
d_0(\mu, \mu_\epsilon) &\le \sqrt{\frac14 (\Omega-\Omega_\epsilon)' \Psi 
(\Omega-\Omega_\epsilon) } 
\le \frac12 \sqrt{\sum_{i=1}^N \sum_{j=1}^N 
(\omega_i-\omega^{\epsilon}_i)(\omega_j - \omega^{\epsilon}_j) \psi_{ij}}  \\
\mathbf E d_0(\mu,\mu_\epsilon) &\le \frac14 \sqrt{   
\sum_{i=1}^N \sum_{j=1}^N \mathbf E (\omega_i-\omega^{\epsilon}_i)(\omega_j - 
\omega^{\epsilon}_j) \psi_{ij} } \\
&= \frac14 \sqrt{  \sum_{i=1}^N \mathbf E (\omega_i-\omega^{\epsilon}_i)^2 
\psi_{ii} } = \frac{\sqrt 2}4 \sqrt{  \sum_{i=1}^N \var(\omega_i) \psi_{ii} }.
\ee
\end{proof}
This quantity will be roughly $m^{-1/2}$ when all $m_i = m$, so the error converges
to zero in the total variation norm at the expected rate.

\subsection{Application to Minibatching Metropolis-Hastings} 
\label{sec:MHMinibatch}
We first make some general observations about $\P_\epsilon$ that arises
from approximating Metropolis-Hastings acceptance ratios.
Consider a generic Metropolis-Hastings algorithm with target measure
$\mu$, proposal kernel $Q(x,\cdot) = Q_x(\cdot)$ with Radon-Nikodym derivative
with respect to $\mu$
\be
\frac{dQ_x}{d\mu}(y) =q(x,y) 
\ee
and Markov transition operator $\P$. Suppose $V$ is a Lyapunov function of $\mathcal{P}$. 
Let
\be \label{eq:AccRatioKappa}
\kappa(x,y) &= \frac{q(y,x)}{q(x,y)}, \quad \alpha(x,y) &= 1 \wedge \kappa(x,y). 
\ee
Then we can write $\P V$ as
\be \label{eq:LyapunovMH}
\P V(x) &= \int V(y) \alpha(x,y) Q(x,dy) + V(x) \left(1- \int \alpha(x,y) Q(x,dy) \right). 
\ee

Let $\P_{\epsilon}$ be the transition kernel of another Metropolis algorithm with the same proposal distribution, but which replaces $\alpha(x,y)$ with $\alpha_{\epsilon}(x,y)$. Let $d$ be a lower semicontinuous metric on $\X$. Then
\be
d(\delta_x \P, \delta_x \P_\epsilon) &= \sup_{\varphi \in \Lip_1(d)} \int_{\X} (\varphi(y)-\varphi(x)) \{\alpha(x,y)-\alpha_\epsilon(x,y)\} Q(x,dy).
\ee
In particular if $d=d_\beta(x,y) = \{2+ \beta(V(x)+V(y))\} \mathbf 1(x \ne y)$ then via \eqref{eq:rho} and the equivalence of $\rho_\beta$ and $d_\beta$ 
\be
d_\beta(\delta_x \P,\delta_x \P_\epsilon) &= \int_{\X} (1+\beta V(y)) |\alpha(x,y)-\alpha_\epsilon(x,y)| Q(x,dy)
\ee
and if $d = d_0(x,y) = 2 \mathbf 1(x \ne y)$ then
\be
d_0(\delta_x \P,\delta_x \P_\epsilon) &=  2 \int |\alpha(x,y)-\alpha_\epsilon(x,y)| Q(x,dy)
\ee
Define $\Delta_\epsilon(x,y) = (\alpha-\alpha_\epsilon)(x,y)$. It follows that to achieve $d(\delta_x \P, \delta_y \P_\epsilon)$ where $d$ is a weighted or ordinary total variation metric, we must have $|\Delta_\epsilon(x,y)|$ small whenever $|x-y|$ is ``small''. Here ``small'' depends on the proposal kernel $Q$ and, in the weighted case, the Lyapunov function. Clearly, the lighter the tails of $q(x,y) V(y) m(y)$, the smaller the neighborhood of $x$ over which we require tight control over $|\Delta_\epsilon(x,y)|$. To know that an approximation is \emph{not} accurate, then, it is sufficient to check that $|\Delta_\epsilon(x,y)|$ is typically \emph{not small} by simulating the Markov chain numerically and computing its value pathwise. Of course this requires that it is possible to simulate from $\P$ in reasonable computing time, which is not always the case. Here, our purpose is to investigate the accuracy of minibatching, so we will focus on an example where it is possible to do this.

We consider a Metropolis-Hastings scheme for logisitic regression, a type of generalized linear model. Thus the unnormalized target for the exact and minibatch-based algorithms take the form in \eqref{eq:GLMlik}, with the specific values of $g^{-1}$ and $A$ given by
\be 
\log \{m(x)\} &\propto \sum_{i=1}^N z_i \langle w_i,x \rangle - \log(1+e^{\langle w_i, x \rangle}) - \frac12 \|B^{-1/2} (x-b)\|^2 \label{eq:LogRegTargets} \\
\log \{m_{\mathcal A}(x)\} &\propto \frac{N}{|\mathcal A|} \sum_{i \in \mathcal A} z_i \langle w_i,x \rangle - \log(1+e^{\langle w_i, x \rangle}) - \frac12 \|B^{-1/2} (x-b)\|^2 \label{eq:LogRegTargetsMinibatch}, 
\ee
where again $\mathcal A$ is a random subset of the integers between 1 and $N$ satisfying $|\mc A| = N_0(\epsilon)$. The value of $N_0(\epsilon)$ is chosen to achieve the desired approximation error. The complete minibatching algorithm that we consider randomly chooses a new $\mc A$ at each iteration. However, we initially consider the properties of the transition kernel $\P_{\mc A}$ that evaluates \eqref{eq:LogRegTargetsMinibatch} at each iteration with a time-invariant subset $\mc A$, which transfer to the randomized algorithm. 

For the remainder of this section, all probability measures of interest have densities with respect to Lebesgue measure. For simplicity, we write the density of $\P(x,\cdot)$ as $p(x,y)$, and the density of the proposal $Q(x,\cdot)$ as $q^*(x,y)$, all with respect to Lebesgue measure. We hope this does not cause any confusion with the Radon-Nikodym derivative of $Q$ with respect to the target $\mu$ denoted by $q(x,y)$.

Throughout, we consider random walk Metropolis with an isotropic Gaussian proposal 
\be \label{eq:GaussianRWMH}
q^*(x,y) = |2 \pi \tau^2 I|^{-1/2} e^{-\|\tau^{-1} (y-x)\|^2/2}
\ee
for $\tau > 0$. We now verify Assumptions \ref{ass:small} and \ref{ass:lyapunov} for $\P$. \citet{jarner2000geometric} show sufficient conditions for random-walk Metropolis to have a Lyapunov function (similar results in one dimension may be found in \citet{mengersen1996rates}). The following corollary combines several results from \cite{jarner2000geometric}.
\begin{corollary} \label{cor:MetropolisLyapunov}
 Suppose $\P$ is defined by a Metropolis algorithm on $\R^p$ with proposal kernel $Q(x,\cdot)$ having density $q^*(x,y)$ with respect to Lebesgue measure satisfying
 \begin{enumerate}
  \item $q^*$ is of ``random walk type'': $q^*(x,y) = q^*(\|x-y\|)$
  \item There exists $\epsilon_q,\delta_q>0$ such that $q^*(x,y) > \epsilon_q$ for $\|x-y\| < \delta_q$ 
  \item For every $x$ we have
  \be \label{eq:ProposalBoundedMoment}
  \int \| x-y \| q^*(x,y) dy < \infty. 
  \ee
 \end{enumerate}
 Suppose further that the target $\mu$ has density $m(x)$ with respect to Lebesgue measure satisfying
 \be
 \limsup_{\|x\| \to \infty} \left\langle \frac{x}{\|x\|}, \nabla \log\{m(x)\} \right\rangle = -\infty \\
 \limsup_{\|x\| \to \infty} \left\langle \frac{x}{\|x\|}, \frac{\nabla m(x)}{\|\nabla m(x)\|} \right \rangle < 0.
 \ee
Then there exist $s,c>0$, $\gamma \in (0,1)$, and $K < \infty$ such that $V(x) = c e^{s\|x\|}$ satisfies $\P V \le \gamma V + K$.
\end{corollary}
\begin{proof}
With the stated conditions on $q$ and $m$ we can immediately apply \citet[Theorem 4.3]{jarner2000geometric}, showing that $\P$ is geometrically ergodic in the sense of equation (11) of \citet{jarner2000geometric}. Now, \citet[Theorem 3.1]{jarner2000geometric} implies that there exists a function $V : \X \to [1,\infty)$, $\gamma \in (0,1)$ and $K < \infty$ such that $(\P V)(x) \le \gamma V(x) + K \1_{\mc S}(x) \le \gamma V(x) + K$ for a set $\mc S \subset \X$. The condition in \eqref{eq:ProposalBoundedMoment} combined with \citet[Theorem 3.3]{jarner2000geometric} implies we can take $V(x) = c e^{s\|x\|}$. 
\end{proof}

We now show Theorem \ref{thm:HM2011} for the Gaussian random walk
Metropolis algorithm under consideration, as well as a Lyapunov
function for the minibatching algorithm that targets
\eqref{eq:LogRegTargetsMinibatch}.

We start by establishing the needed minorization  condition  
 \be \label{eq:MinorizationCompactSets}
 \sup_{x,y \in \mc S} d_0(\delta_x \P, \delta_y \P) \le 2 \bar \alpha_0
 \ee
 for some $\bar \alpha_0 >0$ depending on $\mc S$. 

For this estimate to hold, it is enough for the  transition destiny $p(x,y)$ uniformly bounded from below
 over $x,y \in \mc S$.  Since the invariant measure's density $m(x)$
 and the proposals kernel  $q^*(x,y)$ are continuous in their parameters
 and everywhere positive, we know that $m(x)$ and $q^*(x,y)$ are bounded from above
 and below by positive constants uniformly over any compact
 sets. This implies the desired lower bound on $p(x,y)$  since
\be
 p(x,y) &= \alpha(x,y) q^*(x,y) + \delta_x(y) \int (1-\alpha(x,y)) q^*(x,y) dy \\
 &\ge \alpha(x,y) q^*(x,y) = \left( 1 \wedge \frac{m(y)}{m(x)} \right) q^*(x,y) \\
 \inf_{x,y \in \mc S} p(x,y) &\ge \frac{\inf_{y \in \mc S} m(y)}{\sup_{x \in \mc S} m(x)} \inf_{x,y \in \mc S} q^*(x,y) > 0.
 \ee
 
Combining this with  Assumptions \ref{ass:lyapunov} and
\ref{ass:small} produces the desired result, namely:
\begin{corollary}
 Suppose $\P$ is a Metropolis kernel with proposal density in \eqref{eq:GaussianRWMH} and target satisfying \eqref{eq:LogRegTargets}. Then Theorem \ref{thm:HM2011} holds for $\P$ with Lyapunov function $V(x) = c e^{s \|x\|}$ for $c,s>0$. Further, suppose $\P_{\mc A}$ is a Metropolis kernel with proposal density \eqref{eq:GaussianRWMH} targeting \eqref{eq:LogRegTargetsMinibatch} using a fixed subset $\mathcal A \subset \{1,\ldots,N\}$ of data at every iteration. Then there exist $c_{\mc A}, s_{\mc A}$ such that $V_{\mc A}(x) = c_{\mc A} e^{s_{\mc A} \|x \|}$ is a Lyapunov function of $\P_{\mc A}$.
\end{corollary}
\begin{proof}
 We first show the stated results on the Lypanuov functions. The three conditions on $q$ in Corollary \ref{cor:MetropolisLyapunov} are immediate for the isotropic Gaussian proposal in \eqref{eq:GaussianRWMH}. We now calculate $\nabla \{\log m(x)\}$ and $\nabla m(x)$. Putting $\tilde B = B^{-1}$, we have 
\be
\frac{\partial}{\partial x_j} \log\{m(x)\} &= \left( \sum_{i=1}^N z_i w_{ij} - \frac{w_{ij} e^{w_i x}}{1+e^{w_i x}}\right) - x_j \sum_k \tilde B_{jk} + b_j \sum_k \tilde B_{jk} \\
\frac{\partial}{\partial x_j} m(x) &= m(x) \frac{\partial}{\partial x_j} \log\{m(x)\}
\ee
so with $\tilde z_i(x) = e^{w_i x} \{1+e^{w_i x}\}^{-1} \in [0,1]$, $\{\tilde z(x)\}_i = \tilde z_i(x)$, and $W$ the $N \times p$ matrix with $i$th row $w_i$, we have
\be
\nabla \log\{m(x)\} &= W'z - W'\tilde z(x) - B^{-1} x + B^{-1} b \\
\nabla m(x) &= m(x) (W'z - W'\tilde z(x) - B^{-1} x + B^{-1} b) 
\ee
so that
\be
x' (\nabla \log\{m(x)\}) = x'(W'(z-\tilde z(x))) - x' B^{-1} x + x' B^{-1} b
\ee
and since $z-\tilde z(x) \in [-1,1]$ we have
\be
\limsup_{\|x\| \to \infty} \|x\|^{-1} x'(\nabla \log\{m(x)\}) = -\infty.
\ee
Moreover
\be
x' (\nabla m(x)) &= m(x) x' (\nabla \log\{m(x)\}) \\
\frac{x'(\nabla m(x))}{\|x\| \|\nabla m(x)\|} &= \frac{m(x) x' (\nabla \log\{m(x)\})}{\|x \| m(x) \|\nabla \log \{m(x)\} \|} = \frac{x' (\nabla \log\{m(x)\})}{\|x \| \|\nabla \log \{m(x)\} \|} \\
&= \frac{x'(W'(z-\tilde z(x))) - x' B^{-1} x + x' B^{-1} b}{\|x\| \|x'(W'(z-\tilde z(x))) - x' B^{-1} x + x' B^{-1} b\|}
\ee
so
\be
\limsup_{\|x \| \to \infty} \frac{x'(\nabla m(x))}{\|x\| \|\nabla m(x)\|} &= \limsup_{\|x\| \to \infty} \frac{x'(W'(z-\tilde z(x))) - x' B^{-1} x + x' B^{-1} b}{\|x\| \|W'(z-\tilde z(x)) - B^{-1} x + B^{-1} b\|} \\
&= \limsup_{\|x\| \to \infty} \frac{-x'B^{-1} x}{\|x\| \|B^{-1} x\|} \le \limsup_{\| x \| \to \infty} \frac{-\lambda_{\max}(B)^{-1} \|x\|^2}{\|x\|^2 \lambda_{\min}(B)^{-1}} \\
&= \frac{-\lambda_{\max}(B)^{-1}}{\lambda_{\min}(B)^{-1}} < 0,
\ee
where $\lambda_{\min}(B),\lambda_{\max}(B)$ are the smallest and largest eigenvalues of $B$, both positive since $B$ is positive-definite. So the conditions of Corollary \ref{cor:MetropolisLyapunov} are satisfied for $\P$, and we have that $V(x) = c e^{s\|x\|}$ is a Lyapunov function. The proof of the result for $\P_{\mc A}$ is virtually identical.
\end{proof}

Now we randomize the index subset at each iteration. Let $\mathscr P$ be the set of all minibatch random-walk Metropolis kernels targeting $m_{\mathcal A}(x)$ given by \eqref{eq:LogRegTargets} for some $\mc A$ satisfying $|\mc A| = N_0(\epsilon)$. For any $\P_{\mathcal A} \in \mathscr P$, there exists a Lyapunov function $V_{\mathcal A}(x) = c_{\mathcal A} e^{s_{\mathcal A} \|x\|}$. The set $\mathscr P$ is finite; define
\be
V(x) = \min_{\mc A \subseteq \{1,\ldots,N\} : |\mc A| \ge N_0(\epsilon)} V_{\mathcal A}(x) \ge c e^{s \|x\|}, \quad s = \min_{\mathcal A} s_{\mathcal A}, \quad c = \min_{\mathcal A} c_{\mathcal A},
\ee
so that $c e^{s \|x\|}$ is a Lyapunov function of every $\P_\epsilon$ arising from minibatching with subsets of size at least $N_0$, as well as $\P$, a consequence of the fact that fractional powers of Lyapunov functions are Lyapunov functions by Jensen's inequality. It follows that the Markov kernel defined by
\be
\P_\epsilon(x,\cdot) = \sum_{\mathcal A : |\mc A| = N_0(\epsilon)} \P_{\mathcal A}(x,\cdot) \pi_{\mathcal A}, 
\ee
where $\sum_{\mc A : |\mc A| = N_0(\epsilon)} \pi_{\mathcal A} = 1, 0 \le \pi_{\mathcal A} \le 1$, has Lyapunov function $V(x) = c e^{s \|x\|}$. This is enough to apply Theorem \ref{thm:TVperturb} \emph{if} $\P_\epsilon$ satisfies Assumption \ref{ass:closeness}. Obviously the necessary size $N_0(\epsilon)$ of the batches is a function of the desired approximation error. We now assess empirically how large $N_0$ must be to achieve different levels of approximation error in total variation.

We assess $d_\beta(\delta_x \P, \delta_x \P_\epsilon)$ and $d_0(\delta_x \P,
\delta_x \P_\epsilon)$ by simulating $\P$ as well as $\P_\epsilon$, and computing
$|\Delta_\epsilon(x,y)|$ at each step. To define $\P_\epsilon$, we fix a
subset size $|\mathcal A| = N_0$, and assign equal probability to each
subset of size $N_0$. We fix $N=100,000$ and $p=2$ and consider values of $N_0$
ranging from $1,000$ to $99,000$.We use the adaptive Metropolis algorithm of 
\cite{haario2001adaptive} with the 
scaling factor suggested in \cite{roberts2001optimal} to construct $B$.

Figure \ref{fig:LogregRatio} shows results. We plot $|\Delta|$ as a function of 
the Mahalanobis distance
$D_{\hat \Sigma}(x,\hat x) \equiv (x - \hat x) \hat{\Sigma}^{-1} (x - \hat x)$,
where $\hat x$ and $\hat \Sigma$ are estimates of the posterior mean and 
covariance based on samples of the exact algorithm after discarding a 
burn-in. We also estimate
$\mathbf P(|\Delta_\epsilon(x,y)| < \epsilon)$ 
as a function of $D_{\hat \Sigma}(x,\hat x)$ using local regression (LOESS) 
for $\epsilon = 0.1$. 
Results are shown for the case of independent normal $w_i$ with identity 
covariance. When the current state is near the ``center'' of the state space -- 
that is, close to $\hat x$ with respect to the metric $D_{\hat \Sigma}$ -- $\Delta$ has 
larger mean and the distribution is almost symmetric around 0.5. Similarly, 
the probability of achieving $|\Delta_\epsilon(x,y)| < \epsilon$ decreases as the 
state moves closer to the posterior mean. Naturally, the larger the value of 
$N_0$, the higher the probability of achieving $|\Delta| < \epsilon$, though 
it is notable that more than half the data are necessary to make this 
probability greater than 0.5 in a $D_{\hat \Sigma}$ neighborhood of the mean of 
radius greater than one. This suggests the minibatching strategy will give 
small computational advantage if the goal is to achieve a condition such as Assumption 
\ref{ass:closeness}. These results are generally consistent with those of 
\citet{bardenet2017markov}.

\begin{figure}[h]
\begin{tabular}{cc}
\includegraphics[width=0.5\textwidth]{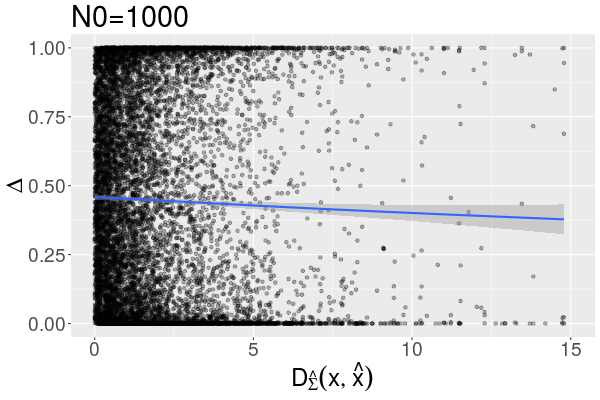} & 
\includegraphics[width=0.5\textwidth]{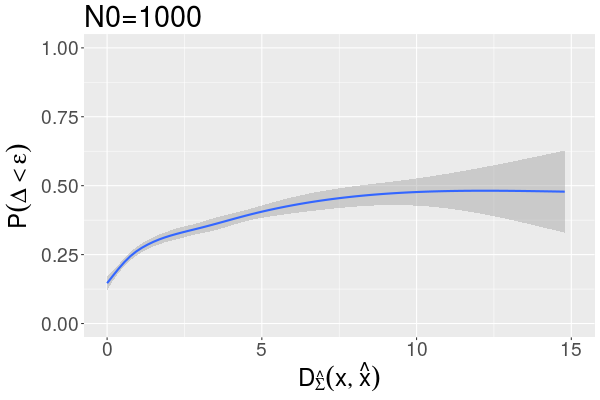} \\
\includegraphics[width=0.5\textwidth]{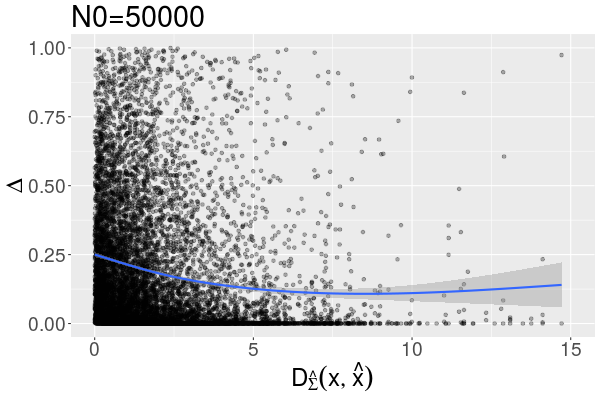} & 
\includegraphics[width=0.5\textwidth]{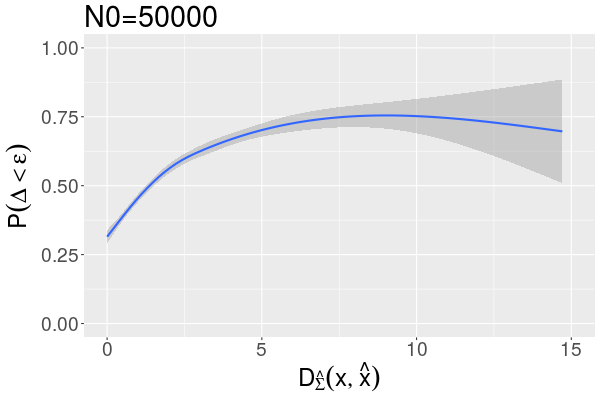} \\
\includegraphics[width=0.5\textwidth]{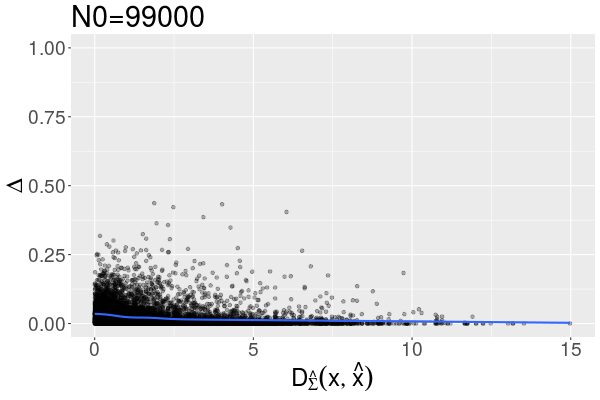} & 
\includegraphics[width=0.5\textwidth]{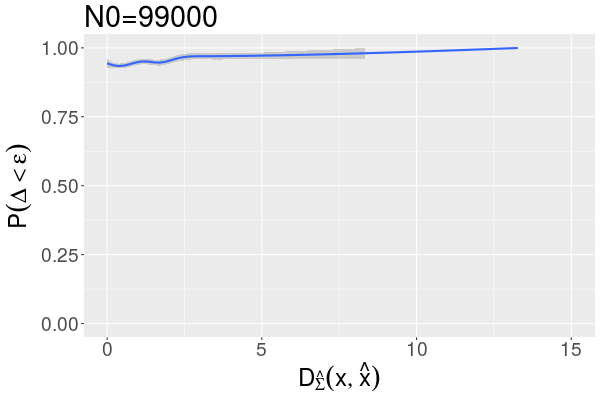} \\
\end{tabular}
\caption{Samples of $\Delta(x,y,\epsilon)$ as a function of $D_{\hat 
\Sigma}(x,\hat x)$ (left column) and estimated $\mathbf P[\Delta(x,y,\epsilon) 
< \epsilon]$ as a function of $D_{\hat 
\Sigma}(x,\hat x)$ (right column) for different values of $N_0$; the function 
estimation uses LOESS local linear smoothing.} \label{fig:LogregRatio} 
\end{figure}

\section{Application to Spatial Statistics} \label{sec:spatial}
\subsection{Gaussian process model and approximations}
We now turn to approximate MCMC for Gaussian process models outlined in
Section \ref{sec:GPSummary}. 
Consider a Gaussian process model with sampling model given by \eqref{eq:GPLikelihood}
and with squared exponential kernel having spatial covariance $\log \Sigma_{ij} = -x_1 \|w_i - w_j\|^2$.
A common prior choice defines $\pi(x) = \pi_1(x_1) \pi_2(x_2) \pi_3(x_3)$ by
\be
\pi(x_3) &= b^a\Gamma(a)^{-1} (x_3^2)^{-\frac{a}2-1} e^{-\frac{b}{2 x_3^2}}, \quad \pi(x_j) = |S_j|^{-1} \mathbf 1\{ x_j \in S_j \}, \quad j = 1,2,
\ee 
with $S_j$ a finite interval that typically does not include zero.
Integration over 
$x_3^2$ is available in closed form. We consider the case where $x_2=1$ is known and we target the posterior
for $x_1$, leading to the sampling model for $z$ given the remaining unknown $x_1=x$
after integrating over $x_3$
\be \label{eq:MLikGP}
L(z,x) \propto |I+\Sigma(x)|^{-\frac12} \Big(b+z'(I+\Sigma(x))^{-1} z \Big)^{-\frac{a+N}2}
\ee
leading to the target
\be
\log \{ m(x) \} \propto -\frac12 \log |I+\Sigma(x)| \frac{a+N}2 \log \Big(b+z'(I+\Sigma(x))^{-1} z \Big) \mathbf 1(x \in S)
\ee
where we put $S_1 = S$ since there is only one unknown.

We define $\P$ by a Metropolis algorithm with a wrapped Gaussian
random walk on the interval $S$ centered at $x$ with variance $v$. Without 
loss of generality, take $|S| = 2\pi$ with midpoint $m$ so that the density of the
proposal with respect to Lebesgue measure is given by
\be \label{eq:PropX1}
q^*(x,y) := \frac1{\sqrt{2 \pi v}} \sum_{k=-\infty}^{\infty}
e^{-\frac{(y-x+2 \pi k)^2}{2 v}} \1 \{m-\pi < y < m+\pi\} \, dy.
\ee
This density can also be expressed using Jacobi theta functions, which
we exploit below.

This algorithm is computationally expensive because it requires
that we compute the determinant of $I+ \Sigma(x)$ and
a quadratic form in its inverse at every step. We consider an 
approximating kernel $\P_\epsilon$ that saves computation by
discretizing the proposal kernel. Observe that
\be
(I+ \Sigma)^{-1} = (I + U \Lambda U')^{-1} = U(I + \Lambda)^{-1}U',
\ee 
so that if we have the spectral decomposition of $\Sigma$ available, we can
easily compute the inverse appearing in \eqref{eq:MLikGP} and its determinant.
Therefore, we discretize $\mathbf X_1$ to a $\epsilon$-grid of points, and 
only propose states on this grid, leading to a modified proposal kernel $Q_\epsilon(x,\cdot)$.
Denote these points as $\{\theta_k\}_{k \in \mathbb N}$. In practice, one 
would pre-compute the spectral decomposition at some small set of support 
points that are likely to be visited frequently by the chain, and then expand 
this set as necessary while the algorithm runs. When $N$ is very large, computing
the likelihood at even one point may be prohibitive; we consider an algorithm 
designed for this setting in the next section.

Define $\P_\epsilon$ by sampling $y^*$ from \eqref{eq:PropX1}, then proposing 
$y = \argmin_{\theta_k} |y^* - \theta_k|$,
the closest support point to $y^*$.
Since $\P_\epsilon$ and $\P$ 
are mutually singular, the weighted total variation bounds we used  
to study approximating kernels for generalized linear models are not useful for this application, and we use 
our bounds in the 1-Wasserstein metric instead.

\subsection{Wasserstein contraction for Metropolis-Hastings}
Geometric convergence in Wasserstein metrics is less well studied than in total variation, so we begin with some sufficient conditions for establishing Assumptions \ref{ass:CouplingBall} and \ref{ass:LipschitzKernel} that are easier to verify, and use these conditions to establish Theorems \ref{thm:UniformWasserstein} and \ref{thm:WassersteinTimeAvgBound} for our application. Recall that we must first establish that $\P$ is a strict contraction. 

In this section, we will assume that the target $\mu$ and the exact proposal kernel $Q(x,\ccdot)$ are absolutely continuous with respect to Lebesgue measure; of course this does not hold for the approximate proposal, which we will consider later. The following condition implies Assumption \ref{ass:CouplingBall} and is easy to show for our application since the state space is compact.
\begin{remark} \label{rem:CouplingBounded}
Let $m(x)$ be the density of $\mu$ with respect to Lebesgue measure, and let $\mathcal B_\delta(z)$ be a ball of diameter $\delta$ with center $z$. 
Suppose that for some $z^*$ and $\delta>0$ one has
\be \label{eq:KernelBounds}
 \inf_{x}  \inf_{z \in \mathcal{B}_\delta(z^*)} q(z,x) = c_0 > 0, \quad\sup_{x}  \sup_{z \in \mathcal{B}_\delta(z^*)}  q(x,z) = c_1 < \infty
\ee
with
$\frac{dQ_x}{d\mu}(z) = q(x,z)$,
and the target
density $m$ satisfies 
\be \label{eq:TargetBounds}
 \inf_{z \in \mathcal{B}_\delta(z^*)} m(z) =C_0 >0
\ee
Then Assumption \ref{ass:CouplingBall} holds for the independence coupling $\Gamma_{x,y}(du,dv) = \P(x,du) \P(y,dv)$.
\end{remark}
\begin{proof}
Clearly
$c_0/c_1 < \alpha(x,z) \le 1$ 
uniformly over $(x,z) \in \X \times \mathcal{B}_\delta(z^*)$. 
Let
 $I_\gamma \subset \mathcal{B}_\delta(z^*)$ be ball of diameter $\gamma$. Then
 \be
 \inf_x \P(x,I_\gamma) \ge |I_\gamma| \inf_{x}  \inf_{z \in
   \mathcal{B}_\delta(z^*)} q(x,z) m(z) \alpha(x,z)  
\ge \gamma C_0 \frac{c_0^2}{c_1} 
 \ee
 Consider the coupling $\Gamma_{x,y}(A_1,A_2) = \P(x,A_1) \P(y,A_2)$. We have
 \begin{align*}
 \inf_{(x,y) \in \X \times \X} \Gamma_{x,y}( (a,b) : |a-b|< \gamma ) 
 \ge&   \inf_{(x,y) \in \X \times \X}   \P(x, 
 \mathcal{B}_\frac{\gamma}2(z^*)) \P(y, \mathcal{B}_\frac{\gamma}2 
 (z^*)) \\
 \ge& \frac{\gamma^2}4 C_0^2 \frac{c_0^4}{c_1^2}   
 \end{align*}
  establishing the result.
\end{proof}

We show these conditions for our application. Define
$M = M(x) \equiv I+\Sigma(x)$.  
The eigenvalues of $M$ satisfy
\be \label{eq:EigenBounds}
\lambda_{\min}(M) &\ge 1, \quad \lambda_{\max}(M) \le 1+N,
\ee
so $\tr(M) \le 2N$ and since $M$ is positive definite
$|M|^{1/N} \le N^{-1} \tr(M)$,
meaning $|M| \le 2^N$. The target
$m(x)$ is given by \eqref{eq:MLikGP}, so
\be
 C_0 \equiv \frac{(b+\|z\|_2^2)^{-(N+a)/2} }{(1+N)^{N/2}} \le m(x) &\le  \left( b+\frac{\|z\|_2^2}{1+N} \right)^{-(N+a)/2} \equiv C_1, \\
c_0 = \frac1{C_1} \frac1{\sqrt{2 \pi v}} e^{-\frac{\pi^2}{2v}} \le q^*(x,y) &\le \frac1{C_0} \frac1{\sqrt{2 \pi v}} = c_1,
\ee
which shows \ref{rem:CouplingBounded} for the Gaussian process application since $q(x,y) = q^*(x,y)/m(y)$. An easily verifiable condition for our example that implies Assumption \ref{ass:LipschitzKernel} is the following
\begin{remark} \label{rem:ContractingMH}
Consider a Metropolis-Hastings algorithm with proposal kernel
$Q(x,\cdot)$ and acceptance probability $\alpha(x,y)$. Recalling the
definition of the metric $d$ from \eqref{eq:ddef}, suppose that
$\alpha(x,y) \in \Lip(d)$
for every $y \in \X$, and that
\be \label{eq:QWass}
d(\delta_x Q,\delta_y Q) \le C_0 d(x,y) \,.
\ee
Then Assumption \ref{ass:LipschitzKernel} holds.
\end{remark}
\begin{proof}
First observe that because $\varphi \in \Lip(d)$ implies that $\varphi$
is bounded we have that
  $\sup_{x \in \X} (Q \varphi)(x) = C_1 < \infty$.
Now  
\be
(\P \varphi)(x-y) 
&= \int \varphi(z) [\alpha(x,z) Q(x,dz)-\alpha(y,z) Q(y,dz)] \label{eq:term1MH} \\
&- \varphi(x) \int \alpha(x,z) Q(x,dz) + \varphi(y) \int \alpha(y,z) Q(y,dz) \label{eq:term2MH}
\ee
Focusing now on \eqref{eq:term1MH}, we have
\be
\eqref{eq:term1MH}  &= \int \varphi(z) (\alpha(x,z)-\alpha(y,z))Q(x,dz) + \int \varphi(z) \alpha(y,z) (Q(x,dz)-Q(y,dz)) \\
&\le \int \varphi(z) |\alpha|_{\Lip(d)} d(x,y) Q(x,dz) + \int \varphi(z) (Q(x,dz)-Q(y,dz)) \\
&\le C_1 |\alpha|_{\Lip(d)} d(x,y) + C_0 d(x,y) 
\ee
Recognizing that if $\alpha(x,z)$ is Lipschitz in its first argument, then so is $\phi(x,z) = \varphi(x) \alpha(x,z)$, and applying a similar argument to the above gives a similar bound for \eqref{eq:term2MH}.
\end{proof}

We now show these conditions for the Gaussian process
application.
The next remark implies that we can work with the ratio of the target densities to verify that $\alpha$ is Lipschitz; the proof is immediate.
\begin{remark} \label{rem:BetaLip}
Define $\kappa(x,y)$ as in \eqref{eq:AccRatioKappa} 
so that
$\kappa(x,y) = \frac{q(y,x)}{q(x,y)}$ 
Then $\kappa \in \Lip(d)$ implies $\alpha \in \Lip(d)$.
\end{remark}
This implies that, for example, it is enough to check that $\kappa(x,y)$ has bounded derivative.
We show that $\sup_x \frac{\partial}{\partial y} \kappa(x,y) < \infty$ in Section \ref{sec:GPApproxDisc}. 
The proof of
$\sup_y \frac{\partial}{\partial x} \kappa(x,y) < \infty$ is similar and omitted.

Finally, we show \eqref{eq:QWass} for the Gaussian process application. Without loss of
generality, take $x < y$. Then we have, using the Jacobi theta representation of $Q$,
\be
d(\delta_x Q,\delta_y Q) &= \sup_{|\varphi|_{\Lip(d)}<1} \int \varphi(z) (Q(x,dz)-Q(y,dz)) \\
&=  \frac1{2\pi} \sup_{|\varphi|_{\Lip(d)}<1} \int_{-\pi-x}^{\pi-x}\varphi(\xi+x) \vartheta_3(-\tfrac{\xi}2,e^{-v/2}) d\xi - \int_{-\pi-y}^{\pi-y} \varphi(\xi+y) \vartheta_3(-\tfrac{\xi}2,e^{-v/2}) d\xi \\
&\le \frac{|x-y|}\delta + \int_{\pi-y}^{\pi-x} \frac1{2\pi} \vartheta_3(-\frac{\xi}2,e^{-v/2}) d\xi - \int_{-\pi-y}^{-\pi-x} \frac1{2\pi} \vartheta_3(-\frac{\xi}2,e^{-v/2}) d\xi \\
&\le \frac{|x-y|}\delta + 2 C |x-y|, 
\ee
where $\vartheta_3$ is the third Jacobi theta function, and the last
step followed because $\vartheta_3(-\frac{\xi}2,e^{-v/2})$ is clearly bounded
since $e^{-v/2} \in (0,1)$.

\subsection{Approximating Kernels for Gaussian Process Models} \label{sec:GPApproxDisc}
Now consider the approximating kernel $\P_\epsilon$ that
has identical $\alpha(x,y)$ to the exact kernel but uses
an approximating $Q_\epsilon$ that only proposes points $\theta_k$
on an $\epsilon$-discretiztion of $\X$.
Define 
$\mathcal I_j = \{ y : \argmin_k |y-\theta_k| =j \}$
and observe that  
$Q_\epsilon = \sum_{k=1}^{\infty} \delta_{\theta_k} Q(x,\mathcal I_k)$.
For any $\varphi$ we have
\be
(\P \varphi - \P_\epsilon \varphi)(x) 
&= \int (\varphi(y) \alpha(x,y) + \varphi(x) \{1-\alpha(x,y)\}) (Q-Q_\epsilon)(x,dy) 
\ee

So we would like to bound on $Q-Q_\epsilon$ in the Wasserstein-$d$ metric. We have for any 
$\varphi \in \Lip_1(d_\beta)$
\be
\int \varphi(y) (Q-Q_\epsilon)(x,dy) &= \sum_k \int_{\mathcal I_k} \varphi(y) (Q-Q_\epsilon)(x,dy) \\
&\le\left| \sum_k \varphi(\theta_k) Q(x,\mathcal I_k)  - \int_{\mathcal I_k} 
(\varphi(\theta_k) + \epsilon) Q(x,dy) \right| \\
&\le \int \epsilon Q(x,dy) = \epsilon.
\ee
It is worth pointing out that so far this argument holds for any $Q_\epsilon$ 
obtained by an $\epsilon$-discretization of the support of $Q$.

Now we need only show that $\alpha(x,y) \in \Lip(d)$. By Remark \ref{rem:BetaLip},
it is enough to show that $\kappa(x,y)$ has uniformly bounded first derivative. 
We have
\begin{remark} \label{rem:BoundedDerivative}
 There exists a constant $C < \infty$ such that
 $\sup_x \frac{\partial}{\partial y} \kappa(x,y) < C$.
\end{remark}
The proof is given in the Appendix.
It follows that if $Q_\epsilon$ is obtained
from an $\epsilon C^{-1}$-discretization of the support of $Q$, then
Assumption \ref{ass:WassersteinApprox} holds.

\subsection{Use of low-rank approximations}
In the previous example, the only source of approximation error was the use
of an approximate proposal $Q_\epsilon$, and it was enough to uniformly 
bound the derivative of $\alpha$ to control the approximation error. In this
section, we consider a variation on the previous algorithm where both an approximate
proposal $Q_\epsilon$ and an approximate acceptance probability $\alpha_\epsilon$ are
used. 

When the number of points $N$ at which the process is sampled is large, it is 
computationally and numerically difficult to compute a spectral decomposition 
of $\Sigma(x,W)$ at even a single point. Therefore in addition to 
discretizing the state space for $x$, it is common to approximate 
$\Sigma(x,W)$ by its partial spectral decomposition
$\Sigma = U \Lambda U' \approx U \Lambda_\epsilon U'$
where $\Lambda_\epsilon$ is a diagonal matrix that is equal to $\Lambda$ in its 
first $r$ diagonal entries and is zero in its remaining diagonal entries. A more
accurate approximation is possible by writing $\Sigma = I + S$ for a low-rank matrix $S$, but
for simplicity we use the standard low-rank approximation. 
The resulting algorithm therefore has both an approximate proposal $Q_\epsilon$, 
where the approximation error arises from discretization, and an approximated 
acceptance probability $\alpha_\epsilon$, where the approximation error arises 
from using a partial spectral decomposition.

The approximate acceptance ratio $\alpha_\epsilon$ can be expressed as 
$\alpha_\epsilon(x,y) = 1 \wedge \frac{\zeta(y)}{\zeta(x)}$ where
\be
\zeta(x) = (b + \sum_{i=1}^r 
(1+\lambda_i(x)^{-1} + N-r)^{\frac{a+N}2 } \prod_{i=1}^r (1+ 
\lambda_i(x))^{1/2} 
\ee
for $\lambda_i(x)$ the $i$th largest eigenvalue of $\Sigma(x)$. For the algorithm
that both discretizes the proposal kernel and uses the low-rank approximation
to approximate the acceptance ratio in defining $\P_\epsilon$, we have
\begin{remark}
 For every $\epsilon>0$, there exists a $C < \infty$ and $r_\epsilon(x,y) \le N$ such that
 if a rank $r_\epsilon(x,y)$ approximation to $\Sigma(x)$ and $\Sigma(y)$ is used to compute
 $\alpha_\epsilon(x,y)$, the resulting $\P_\epsilon$  
 will achieve Assumption \ref{ass:WassersteinApprox} when $Q_\epsilon$ uses a $C^{-1} \epsilon$-discretization of $\X$.
\end{remark}
\begin{proof}
For any $\varphi$ we have
\be
(\P \varphi - \P_\epsilon \varphi)(x) &= \int \varphi(y) \alpha(x,y) Q(x,dy) + \int 
\varphi(x) (1-\alpha(x,y)) Q(x,dy) \\
&- \int \varphi(y) \alpha_\epsilon(x,y) Q_\epsilon(x,dy) - \int \varphi(x) 
(1-\alpha_\epsilon(x,y)) 
Q_\epsilon(x,dy),
\ee
and adding and subtracting, we get
\be \label{eq:AlphaQ}
\alpha Q - \alpha_\epsilon Q_\epsilon = \alpha (Q- Q_\epsilon) + (\alpha-\alpha_\epsilon) Q_\epsilon. 
\ee
We already know how to deal with the first term, so it remains to handle the second term. For $\varphi \in \Lip_1(d)$, we need
\be
\int \varphi(y) (\alpha-\alpha_\epsilon)(x,y) Q_\epsilon(x,dy) \le \epsilon,
\ee
which depends on how well $\alpha$ approximates $\alpha_\epsilon$, rather than how well $Q$ approximates $Q_\epsilon$. If $\varphi \in \Lip_1(d)$, then $|\varphi|_\infty < 1$, so
\be
\int \varphi(y) (\alpha-\alpha_\epsilon)(x,y) Q_\epsilon(x,dy) 
&\le \int (\alpha-\alpha_\epsilon)(x,y) Q_\epsilon(x,dy), 
\ee
so we need only make the integral on the right side small. We have
\be
\int (\alpha-\alpha_\epsilon)(x,y) Q_\epsilon(x,dy) = \sum_k (\alpha-\alpha_\epsilon)(x,\theta_k) Q_\epsilon(x,\theta_k), 
\ee
so the desired bound will follow if
\be \label{eq:GPApproxCondition}
\sup_{\theta_k} (\alpha-\alpha_\epsilon)(x,\theta_k) < \epsilon.
\ee
It is always possible to make $\epsilon = 0$ by putting $r=N$, though naturally this would eliminate any computational advantage. Regardless, it is clear that for every $\epsilon$ and every $x,y$ there exists $r_\epsilon(x,y) \le N$ such that for $r \ge r_\epsilon(x,y)$ we have 
$(\alpha-\alpha_\epsilon)(x,y) < \epsilon$.
\end{proof}
Evidently, by choosing the rank of the partial spectral decomposition in an adaptive way depending on the state, the proposal, and the desired approximation error, we can achieve \eqref{eq:GPApproxCondition}. Numerical experiments showing that this approximation can be very accurate in some cases using $r \ll N$ can be found in \cite{johndrow2017coupling}.

\subsection{Approximation on an unbounded state space}
We conclude consideration of the Gaussian process example by considering the case where the state space is unbounded. Consider the model in \eqref{eq:GPLikelihood} with $x_2=x_3=1$ known. We parametrize the model in terms of the remaining unknown as
\be \label{eq:GPLikelihoodUnbounded}
L(z, x) &=  e^{-\frac12 \log |2\pi (I+\Sigma(x))| - \frac12 z'(I+\Sigma(x))^{-1} z} \equiv e^{-\Phi(x,z)} \\
\{\Sigma(x)\}_{ij} &= e^{-x^2 \|w_i - w_j\|^2_2}
\ee
and place a standard Gaussian prior on $x$ so that the target density satisfies
$m(x) \propto L(z, x) e^{-\frac{x^2}{2}}$. Because $z$ is not a state variable, we often write
$\Phi(x)$ in lieu of $\Phi(x,z)$. As above, we consider $\P_\epsilon$ that 
both discretizes $\X$ and uses a low-rank approximation to $\Sigma(x)$. 
Here, the state space $\X = \mathbb R$ and we consider Metropolis-Hastings with proposal
\be
y = \phi x + \sqrt{1-\phi^2} \xi, \quad \xi \sim N(0,1)
\ee
for $\phi \in (-1,1)$. This is a simple example of the \emph{pre-conditioned Crank-Nicolson} algorithm studied in \citet{hairer2014spectral}, wherein it is shown that under fairly general conditions this Markov chain satisfies Theorem \ref{thm:WtdWasserstein} with $d(x,y) = 1 \wedge |x-y|$. The acceptance ratios for the pCN algorithm are given by
\be
\alpha(x,y) = 1 \wedge e^{\Phi(x) - \Phi(y)}.
\ee
The key requirements to show the weak Harris theorem (Theorem \ref{thm:WtdWasserstein}) are a Lipschitz condition for $\Phi$ and that the acceptance ratios can be bounded from below in a neighborhood of the current state. The following two lemmas verify these conditions. 
\begin{lemma}
 The function $\Phi$ is globally Lipschitz.
\end{lemma}
\begin{proof}
 We differentiate $\Phi$ to obtain
 \be
 \frac{\partial}{\partial x} \Phi(x;z)&= \frac12 \frac{|2\pi M(x)| \tr((2\pi M(x))^{-1} D(x))}{|2\pi M(x)|} - \frac12 z' M(x)^{-1} D(x) M(x)^{-1} \\
 &= \frac12 \tr((2\pi M(x))^{-1} D(x)) - \frac12 z' M(x)^{-1} D(x) M(x)^{-1} z.
 \ee
 Since
 \be
\{D(x)\}_{ij} &= -2 x \| w_i-w_j\|_2^2 e^{-x^2 \| w_i - w_j \|_2^2}
\ee
so that
\be
\|D(x)\|_F^2 = \sum_i \sum_j 4 x^2 \|w_i-w_j\|_2^4 e^{-2 x^2 \| w_i - w_j\|_2^2}.
\ee
The function $4 x^2 \delta e^{-2 x^2 \delta}$ has maxima at $\pm \frac1{\sqrt{2 \delta}}$ and a minimum at zero, is bounded, and converges to zero at $x = \pm \infty$. It follows that the entries of $D$ are uniformly bounded, so there exists $\bar D$ such that $\|D(x)\|_F^2<\bar D < \infty$. Clearly $M(x) = I + \Sigma(x)$ has eigenvalues bounded away from zero and infinity. It follows that the derivative of $\Phi$ is uniformly bounded.
\end{proof}

\begin{lemma}
 There exist constants $-\infty<c<C<\infty$ such that $c < \inf_x \Phi(x) < \sup_x \Phi(x) < C$. 
\end{lemma}
\begin{proof}
 Since $0 \le e^{-x^2 \delta} \le 1$ for all nonnegative $\delta$, and $\Sigma(x)$ is positive-definite, the eigenvalues of $\Sigma(x)$ satisfy 
 \be
 0 \le \inf_x \lambda_{\min}(\Sigma(x)) \le \sup_x \lambda_{\max}(\Sigma(x)) \le N.
 \ee
 So $\log |I+\Sigma(x)|$ is bounded below by zero and above by a finite constant. We also have
 \be
 \|z\|_2^2 (1+N)^{-1} \le z'(I+\Sigma(x))^{-1} z \le \|z\|_2^2.
 \ee
 The result follows. 
\end{proof}

The two lemmas give $\tilde d_1(\delta_x \P^n, \delta_y \P^n) \le \bar \alpha \tilde d_1(x,y)$ using the conditions in \citet{hairer2014spectral}. It follows that there exists a $\beta > 0$ such that $\tilde d_\beta(\delta_x \P, \delta_y \P) \le \bar \alpha \tilde d_\beta(x,y)$, which is essentially a consequence of the equivalence of the conditions needed for Theorem \ref{thm:WtdWasserstein} and the conditions used in \citet{hairer2014spectral}.

We now show the following error condition.
\begin{theorem}
For any $\epsilon > 0$ there exists $r(\epsilon) \le N$ and $\epsilon_0$ such that if $\P_\epsilon$ approximates the pCN Markov operator $\P$ by using an $\epsilon_0$-discretization of $\X$ and approximating $\alpha$ by $\alpha_\epsilon$ using $r(\epsilon)$ eigenvectors, we have
 \be
\tilde d_\beta(\delta_x \P, \delta_x \P_\epsilon) \le \epsilon(1+ \sqrt{\beta V(x)}).
 \ee
\end{theorem}

\begin{proof}
 We construct a coupling of $Y \sim \delta_x \P$ and $Y_\epsilon \sim \delta_x \P_\epsilon$ as follows. Let
 \be
 y = \theta x + \xi, \quad y_\epsilon = \theta x +\xi^*,
 \ee
 where $\xi^*$ is the nearest point to $\xi$ in a $\epsilon_0$ discretization of $\X$. Now let $\zeta \sim \text{Uniform}(0,1)$ and accept $y$ if $\zeta < \alpha(x,y)$ and accept $y_\epsilon$ if $\zeta < \alpha_\epsilon(x,y_\epsilon)$. Then, using the fact that the Lyapunov function $V(x)$ is continuous and grows at most exponentially in $x$ (see \citet[Lemma 3.2]{hairer2014spectral} for this latter condition), we have
 \be
 \E[\tilde d^2_\beta(Y,Y_\epsilon) \mid \xi ] &\le \epsilon_0 (2 + \beta V(\theta x+|\xi|) + \beta V(\theta x + |\xi| + \epsilon_0)) \mathbf P(\text{both accept} \mid \xi)\\
 &+ (1 \wedge |\xi+\epsilon_0|)  (2 + \beta V(x) + \beta V(\theta x +
 |\xi| + \epsilon_0)) \mathbf P(\text{one accepts} \mid \xi) \\
&+ \tilde d^2_\beta(x,x)  \mathbf P(\text{both reject} \mid \xi)\\
 &\le \epsilon_0 (2 + e^{|\xi|} \beta V(\theta x) + e^{|\xi+\epsilon_0|} \beta V(\theta x)) \\
 &+ (1 \wedge |\xi+\epsilon_0|) (2 + \beta V(x) + e^{|\xi+\epsilon_0|} \beta V(\theta x)) \mathbf P(\text{one accepts} \mid \xi).
 \ee
 Now since $1 \wedge e$ is Lipschitz-1 with respect to $|\cdot|$,
 \be
 \mathbf P(\text{one accepts} \mid \xi) = \mathbf P[\zeta \text{ between } \alpha_\epsilon, \alpha \mid \xi] \le 1 \wedge |\Phi(\theta x + \xi) - \Phi_\epsilon(\theta x + \xi^*)|  .
 \ee
 Choose $r(\epsilon), \epsilon_0$ so that
 \be
 \frac{\epsilon^2}4 &> \epsilon_0 \E[ e^{|\xi + \epsilon_0|} (1 \wedge |\xi+\epsilon_0|)  ],\\
 \frac{\epsilon^2}4 &> \sup_{|\xi-\xi^*| \le \epsilon_0} \sup_x \E[ e^{|\xi + \epsilon_0|} (1 \wedge |\xi+\epsilon_0|) (1 \wedge |\Phi(\theta x + \xi) - \Phi_\epsilon(\theta x + \xi^*)|)],
 \ee 
 which is always possible since the expectation in the first line is finite, and because $\Phi$ and $\Phi_\epsilon$ are Lipschitz we can make the second quantity arbitrarily small by taking $r(\epsilon) \to N, \epsilon_0 \to 0$. Then
 \be
 \E[\tilde d^2_\beta(Y,Y_\epsilon) ] &\le \epsilon_0 \E[2 + e^{|\xi|} \beta V(\theta x) + e^{|\xi+\epsilon_0|} \beta V(\theta x)] \\
 &+ (2 + \beta V(x)) \E[(1 \wedge |\xi+\epsilon_0|) (1 \wedge |\Phi(\theta x + \xi) - \Phi_\epsilon(\theta x + \xi^*)|)] \\
 &+ \beta V(\theta x) \E[e^{|\xi+\epsilon_0|} (1 \wedge |\xi+\epsilon_0|) (1 \wedge |\Phi(\theta x + \xi) - \Phi_\epsilon(\theta x + \xi^*)|)  ] \\
 &\le \frac{\epsilon^2}4 (2 + 2 \beta V(x)) + \frac{\epsilon^2}4 (2 + \beta V(x)) + \frac{\epsilon^2}4 \beta V(x). \label{eq:pCNError}
 \ee
 So we obtain using Jensen's inequality and triangle inequality
 \be
 \E[\tilde d^2_\beta(Y,Y_\epsilon) ] &\le \epsilon^2 (1+\beta V(x)) \\
 \E[\tilde d_\beta(Y,Y_\epsilon) ] & \le \epsilon (1 + \sqrt{\beta V(x)}).
 \ee
\end{proof}

Finally, we show that $\P_\epsilon$ and $\P$ have a common Lyapunov function when the discretization is sufficiently fine. First note that the kernel $\P_\epsilon$ that differs from $\P$ only in the use of a truncated eigenfunction expansion is just a special case of the kernel $\P_m$ in \citet{hairer2014spectral}. Thus $\P_\epsilon$  has the same Lyapunov function as $\P$ by \cite[Lemma 3.2]{hairer2014spectral}, which we can take to be continuous and grow no faster than an exponential in $x$. Thus it remains to show that preservation of the Lyapunov function by discretization. We have
\be
(\P_\epsilon V)(x) &= \sum_k V(\theta_k) \P(x,I_k) \le \sum_k e^{\epsilon} \int_{I_k} V(y) \P(x,dy) \\
&\le e^{\epsilon} \sum_k \int_{I_k} V(y) \P(x,dy) \le e^{\epsilon} (\P V)(x) \le e^{\epsilon}(\gamma V(x) +K).
\ee
Thus taking $\epsilon$ sufficiently small that $\gamma e^{\epsilon}< 1$, we obtain the desired result. Error bounds can now be obtained by application of Theorem \ref{thm:WtdWErrorBound}.

\appendix

\section{Additional proofs}

\subsection{Proof of Theorem \ref{thm:ErrorBound}} 
The following calculation follows the spirit of \cite{GlynnMeyn1996, KontoyiannisMeyn2003}.
For any $\phi$ with $\phi \leq V^{1/2}$ define 
$\widetilde \phi = \phi - \mu \phi$
\begin{align}
  \label{eq:1}
  U(x) = \sum_{k=0}^\infty \P^k \widetilde \phi
\end{align}

Now for $p \in (0,1]$ we have $\gamma_p \in (0,1)$ and $K_p>0$ so that
\begin{align}
  \label{eq:2}
  \P V^p (x) \leq \gamma_p V^p(x)+ K_p
\end{align}
and
$\lvertiii \P \widetilde{\phi} \rvertiii_{\beta_p} \leq \alpha_p \lvertiii \widetilde{\phi} \rvertiii_{\beta_p}$ 
is the weighted TV norm built on $V^p$ with an appropriate $\beta_p$. 
Now observe that
\begin{align*}
  \lvertiii U \rvertiii_{\beta_p} \leq  \sum_{k=0}^\infty \alpha_p^k \lvertiii \widetilde{\phi} \rvertiii_{\beta_p} =   \frac{\lvertiii \widetilde{\phi} \rvertiii_{\beta_p}}{1- \alpha_p}
\end{align*}
with $p=\frac12$. 
Observe that since $\phi < V^{\frac12}$, we have $\mu \phi < \mu V^{\frac12} < \infty$, so $|\widetilde \phi| < \mu V^{\frac12} + V^{\frac12}$, and
\be
| U(x) | &\le (\mu V^{\frac12}+ V^{\frac12}(x)) \frac{1}{1-\alpha_{(1/2)}} \le C(1+V^{\frac12}(x))
\ee
for $C = \frac{1 \vee \mu V^{\frac12}}{1-\alpha_{(1/2)}}$.  
This implies that 
\be \label{eq:PoissonBound}
  |U(x)| \leq C( 1 +  V^\frac12(x)).
\ee

Note that
\be \label{eq:Poisson}
  (\P-I)U(x) = - \widetilde{\phi}(x)
\ee
so
\begin{align*}
  U(X_n^\epsilon) -U(X_0^\epsilon)&= \sum_{k=0}^{n-1} U(X_{k+1}^\epsilon) -U(X_k^\epsilon) 
  = \sum_{k=0}^{n-1}  [U(X_{k+1}^\epsilon) -\P_\epsilon U(X_k^\epsilon) ]+ \sum_{k=0}^{n-1} (\P_\epsilon -I) U(X_k^\epsilon) \\
  & = \sum_{k=0}^{n-1}  [U(X_{k+1}^\epsilon) -\P_\epsilon U(X_k^\epsilon) ]+ \sum_{k=0}^{n-1} (\P -I) U(X_k^\epsilon) + \sum_{k=0}^{n-1} (\P_\epsilon -\P) U(X_k^\epsilon) 
\end{align*}
Using \eqref{eq:Poisson} and defining the Martingale $m_{k+1}^\epsilon = U(X_{k+1}^\epsilon) -\P_\epsilon U(X_k^\epsilon) $ and $M_n^\epsilon=\sum_{k=1}^n m_{k}^\epsilon$,
we have
\be \label{eq:PoissonEquation}
         \frac1n \sum_{k=0}^{n-1} \phi(X_k^\epsilon) - \mu \phi   = \frac{U(X_0^\epsilon) -U(X_n^\epsilon)}{n}  + \frac1n M_n^\epsilon +   \frac1n\sum_{k=0}^{n-1} (\P_\epsilon -\P) U(X_k^\epsilon) 
\ee
Now 
\be 
\E[ (m_{k+1}^\epsilon)^2 \mid \mathcal{F}_k] &\leq \P_\epsilon (U^2)(X_k^\epsilon) - [\P_\epsilon (U)(X_k^\epsilon) ]^2 \\
\E \left[ \left( \frac1n M_n^\epsilon \right)^2 \right] &\leq \frac1{n^2} \sum_{k=1}^{n} \mathbf{E}[(m_k^\epsilon)^2], \label{eq:MartingaleBound}
\ee
and it follows from \eqref{eq:PoissonBound} that
$U^2(x) \leq 2 C^2(1+V(x))$.
So then with $X_0^{\epsilon} = x_0$
\be
\mathbf E [\P_{\epsilon}(U^2)(X_k^{\epsilon})] &\le \P_{\epsilon} 2C^2 (1+ \P_{\epsilon}^k V(x_0))  \le 2C^2+ 2C^2 \P_{\epsilon}^{k+1} V(x_0).
\ee

We proceed by bounding the square of each term on the 
right side of \eqref{eq:PoissonEquation}. We have
\be \label{eq:UsualVBound}
(\P^{k+1}_{\epsilon} V)(x_0) &\le \gamma_{\epsilon}^{k+1} V(x_0) + \frac{K_\epsilon}{1-\gamma_\epsilon} \\
\ee
so
\be
\sum_{k=0}^{n-1} \mathbf E[(m_{k+1}^{\epsilon})^2] &\le 2 C^2 
\left( n + \frac{n K_\epsilon}{1-\gamma_\epsilon} + 
\frac{1-\gamma_\epsilon^n}{1-\gamma_\epsilon} V(x_0) \right) \\
\frac{1}{n^2} \sum_{k=0}^{n-1} \mathbf 
E[(m_{k+1}^{\epsilon})^2]  &\le 2 C^2 \left( \frac1n +
\frac{K_\epsilon}{n\{1-\gamma_\epsilon\}} + \frac{1-\gamma_\epsilon^n}{n^2 
\{1-\gamma_\epsilon\}} V(x_0) \right) .
\ee
where we used \eqref{eq:UsualVBound} and \eqref{eq:MartingaleBound} in the above.

Now for the term
$\frac1n\sum_{k=0}^{n-1} (\P_\epsilon -\P) U(X_k^\epsilon)$.
Since $C^{-1} |U| < 1+V^{\frac12}$, for $|\phi| < 1+V$
\be
(\P_{\epsilon}-\P)(\phi) \le \epsilon (1+\delta V) 
\ee
by Jensen's inequality
\be
(\P_{\epsilon}-\P)(\phi^{\frac12}) \le \sqrt{\epsilon (1+\delta V)} \le \sqrt{\epsilon} + \sqrt{\epsilon \delta V} \le \sqrt{\epsilon}(1+\delta^{\frac12} V^{\frac12})
\ee
so we have
\be
C^{-1} |(\P_\epsilon -\P) U(x)| \le \epsilon^{\frac12}(1 + \delta^{\frac12} V^{\frac12}(x)).
\ee
Using these inequalities, we now bound the expectation of
\be \label{eq:ErrorProductTerms}
(\P_{\epsilon} -\P) U(X_k^\epsilon) (\P_{\epsilon}-\P) U(X_j^{\epsilon}).
\ee
Taking $k\ge j$, we get
\be
\eqref{eq:ErrorProductTerms} &\le C^2 \epsilon (1 + \delta^{\frac12} V^{\frac12}(X_k^{\epsilon}))(1+\delta^{\frac12} V^{\frac12}(X_j^{\epsilon})) \\
\mathbf E \left[ \eqref{eq:ErrorProductTerms} \right] &\le C^2 \epsilon \mathbf E \left[ \mathbf E \left[ (1 + \delta^{\frac12} V^{\frac12}(X_k^{\epsilon})) \mid \mathcal{F}_j \right] (1+\delta^{\frac12} V^{\frac12}(X_j^{\epsilon})) \right] \\
&\le C^2 \epsilon \mathbf E\left[ (1 + \delta^{\frac12}  (\P^{k-j}_{\epsilon} V^{\frac12})(X_j^{\epsilon})) (1 + \delta^{\frac12} V^{\frac12}(X_j^{\epsilon})) \right] \\
&\le C^2 \epsilon \mathbf E\left[ \left(1 + \delta^{\frac12} \left( \gamma_\epsilon^{(k-j)/2} V^{\frac12}(X_j^\epsilon) + \frac{\sqrt{K_\epsilon}}{1-\sqrt{\gamma_\epsilon}}  \right) \right) (1 + \delta^{\frac12} V^{\frac12}(X_j^{\epsilon})) \right] \\
&\le C^2 \epsilon \mathbf E\left[ \left(1 + \delta^{\frac12} V^{\frac12}(X_j^\epsilon) + \frac{\sqrt{\delta K_\epsilon}}{1-\sqrt{\gamma_\epsilon}} \right) (1 + \delta^{\frac12} V^{\frac12}(X_j^{\epsilon})) \right] \\
&\le C^2 \epsilon \mathbf E\left[ (2 + 2 \delta V(X_j^\epsilon)) + \frac{\sqrt{\delta K_\epsilon}}{1-\sqrt{\gamma_\epsilon}} (1 + \delta^{\frac12} V^{\frac12}(X_j^{\epsilon})) \right] \\
&\le C^2 \epsilon \left[ 2 + 2 \delta \left(\gamma_\epsilon^j V(x_0) + \frac{K_\epsilon}{1-\gamma_\epsilon} \right) + \frac{\sqrt{\delta K_\epsilon}}{1-\sqrt{\gamma_\epsilon}} \left(1 + \delta^{\frac12} \left(\gamma_\epsilon^{j/2} V^{\frac12}(x_0) + \frac{\sqrt{K_\epsilon}}{1-\sqrt{\gamma_\epsilon}} \right) \right) \right] \\
&\le C^2 \epsilon \left[ \left( 2 + 2 \delta \frac{K_\epsilon}{1-\gamma_\epsilon} + \frac{\sqrt{\delta K_\epsilon}}{1-\sqrt{\gamma_\epsilon}} + \frac{\delta K_\epsilon}{(1-\sqrt{\gamma_\epsilon})^2} \right) +  2 \delta \gamma_\epsilon^j V(x_0) + \delta \frac{\sqrt{K_\epsilon}}{1-\sqrt{\gamma_\epsilon}} \gamma_\epsilon^{j/2} V^{\frac12}(x_0) \right] \\
&\le C^2 \epsilon \left[ \left( 2 + 2 \delta \frac{K_\epsilon}{1-\gamma_\epsilon} + (1+\sqrt{\delta}) \frac{\sqrt{\delta K_\epsilon}}{1-\sqrt{\gamma_\epsilon}} + \frac{\delta K_\epsilon}{(1-\sqrt{\gamma_\epsilon})^2} \right) +   \delta \left( 2 + \frac{\sqrt{K_\epsilon}}{1-\sqrt{\gamma_\epsilon}}  \right) \gamma_\epsilon^{j/2} V(x_0) \right] \\
&\equiv C^2 \epsilon \left[ c_0 +   \delta c_1 \gamma_\epsilon^{j/2} V(x_0) \right], \label{eq:ErrorExpectation} 
\ee
where in various places we used 
\be
(\P_\epsilon V^{\frac12})(x) \le \sqrt{(\P_\epsilon V)(x)} \le  \sqrt{\gamma_\epsilon V(x) + K_\epsilon} \le \sqrt{\gamma_\epsilon} V^{\frac12}(x) + \sqrt{K_\epsilon},
\ee
that $V^{\frac12} \le 1+ V$, and that $(1+\delta^{\frac12} V^{\frac12}(x))^2 \le 2 + 2 \delta V(x)$.
Note that we can bound $c_0$ as
\be
c_0 \le 2 + 5 \frac{(\delta \vee \sqrt{\delta})(K_\epsilon \vee \sqrt{K_\epsilon})}{(1-\sqrt{\gamma_\epsilon})^2}.
\ee
Observe that for $j \ge k$, we obtain the bound in \eqref{eq:ErrorExpectation} with $j$ replaced by $k$. So we get
\be
\sum_{k=0}^{n-1} \sum_{j=0}^{n-1} \mathbf E \left[ (\P_{\epsilon}-\P) U(X_k^{\epsilon}) (\P_{\epsilon}-\P) U(X_j^{\epsilon}) \right] &\le n C^2 \epsilon \left( n c_0 + \delta c_1 \frac{1-(\sqrt{\gamma_\epsilon})^n}{1-\sqrt{\gamma_\epsilon}} V(x_0)  \right) \\
\frac{1}{n^2} \sum_{k=0}^{n-1} \sum_{j=0}^{n-1} \mathbf E \left[ (\P_{\epsilon}-\P) U(X_k^{\epsilon}) (\P_{\epsilon}-\P) U(X_j^{\epsilon}) \right] &\le C^2 \epsilon \left( c_0 + \frac{\delta c_1}{n} \frac{V(x_0)}{1-\sqrt{\gamma_\epsilon}}  \right)
\ee
Finally we have
\be
\frac{(U(X_0^\epsilon) -U(X_n^\epsilon))^2}{n^2} &\le \frac{2 U^2(X_0^\epsilon) + 2 U^2(X_n^\epsilon)}{n^2} \\
\mathbf E \frac{(U(X_0^\epsilon) -U(X_n^\epsilon))^2}{n^2} &\le \frac{4C^2}{n^2} \left( \mathbf E[1+V(X_0^\epsilon)] + \mathbf E [1+V(X_n^\epsilon)] \right) \\
&\le \frac{4C^2}{n^2} \left(1+V(x_0) + \gamma_{\epsilon}^n V(x_0) + \frac{1-\gamma_{\epsilon}^n}{1-\gamma_{\epsilon}} K_{\epsilon} \right) \\
&\le \frac{4C^2}{n^2} \left( 1 + (1+\gamma_\epsilon^n) V(x_0) + \frac{K_{\epsilon}}{1-\gamma_{\epsilon}} \right) 
\ee
Giving us
\be
\mathbf E \left( \frac1n \sum_{k=0}^{n-1} \phi(X_k^\epsilon) - \mu \phi  
\right)^2 &\le  6 C^2 \left(\frac1n +\frac{K_\epsilon}{n\{1-\gamma_\epsilon\}} + 
\frac{1-\gamma_\epsilon^n}{n^2 \{1-\gamma_\epsilon\}} V(x_0) \right) \\
&+ 3 C^2 \epsilon \left( c_0 + \frac{\delta c_1}{n} \frac{V(x_0)}{1-\sqrt{\gamma_\epsilon}}  \right) \\
&+ \frac{12 C^2}{n^2} \left( 1 + (1+\gamma_\epsilon^n) V(x_0) + 
\frac{K_{\epsilon}}{1-\gamma_{\epsilon}} \right) \\
&\le 3 C^2 \epsilon c_0 + \frac{3 C^2}{n} \left(2+ \frac{2K_{\epsilon}}{1-\gamma_\epsilon} + \frac{ \epsilon \delta c_1 V(x_0)}{1-\sqrt{\gamma_\epsilon}} \right) + \mathcal{O}\left( \frac{1}{n^2} \right)
\ee
concluding the proof of Theorem \ref{thm:ErrorBound}.

\subsection{Proof of Proposition \ref{prop:LyapunovProbit}}
The Lyapunov function here is also used in \cite{roy2007convergence}, but we use different estimates for the constants and also show the result for $\P_\epsilon$. 

Consider standard binomial probit
$z_i \sim \Binom(m_i,p_i)$
where $p_i = \Phi(w_i x)$. Consider
\be
V(x) = x'(W'DW)x, \quad D = \diag(m_1,\ldots,m_N).
\ee
Then
\be
\E[V(x^*) \mid (x^*,\Omega)] &= \E[\E[V(x^*) \mid \Omega] \mid x] \\
&= \tr(W'DW(W'DW)^{-1}) + \E[x^* \mid \Omega] (W'DW) \E[x^* \mid \Omega] \\
&= p + \Omega' W (W'DW)^{-1} (W'DW) (W'DW)^{-1} W' \Omega \\
&= p + \Omega' W (W'DW)^{-1} W' \Omega 
\ee
so then
\be
\E[V(x^*) \mid x] &= p+\E[\Omega' W (W'DW)^{-1} W' \Omega \mid x] \\
&= p+\E[\Omega' D^{-1/2} D^{1/2} W(WDW)^{-1} W' D^{1/2} D^{-1/2} \Omega \mid x] \\
&\le p + \E[\Omega'D^{-1} \Omega \mid x] = p + \sum_i \E[\omega_i^2/m_i \mid x].
\ee

Now putting $\xi_i = w_i x$
\be
\E[\omega_i] = z_i \left(\xi_i + \frac{\phi(-\xi_i)}{1-\Phi(-\xi_i)}\right) + (m_i - z_i) \left( \xi_i - \frac{\phi(\xi_i)}{1-\Phi(\xi_i)} \right)
\ee
and
\be
\var[\omega_i] = z_i \left\{ 1 - \frac{\xi_i \phi(-\xi_i)}{1-\Phi(-\xi_i)} - \left(\frac{\phi(-\xi_i)}{1-\Phi(-\xi_i)} \right)^2 \right\} + (m_i - z_i) \left\{ 1 + \frac{\xi_i \phi(\xi_i)}{1-\Phi(\xi_i)} - \left( \frac{\phi(\xi_i)}{1-\Phi(\xi_i)} \right)^2 \right\}
\ee

We use the inequality for $\xi \ge 0$
\be
\frac{2}{\xi+\sqrt{\xi^2+4}} < \frac{1-\Phi(\xi)}{\phi(\xi)} \le \frac{2}{\xi+\sqrt{\xi^2+8/\pi}},
\ee
which implies there exists a function $h(\xi)$ uniformly bounded above by $\sqrt{8/\pi}$ and below by 0 such that 
\be
r(\xi) \equiv \frac{\phi(\xi)}{1-\Phi(\xi)} = \xi+h(\xi)
\ee
for all $\xi \ge 0$ and so for $\xi<0$ we have
$r(-\xi) = -\xi+h(-\xi)$.
Finally, for any $\xi<0$,
$r(\xi) \le 2 \phi(\xi) \le 2$
while for $\xi \ge 0$,  
$r(-\xi) \le 2 \phi(\xi) \le 2$.
So now if $\xi_i > 0$
\be
\E[\omega_i] &= z_i ( \xi_i + r(-\xi_i)) + (m_i - z_i) (\xi_i - (\xi_i + h(\xi_i))) \\
\var[\omega_i] &= z_i \left\{ 1 - \xi_i r(-\xi_i) - r(-\xi_i)^2 \right\} + (m_i - z_i) \left\{ 1 + \xi_i(\xi_i + h(\xi_i)) - \left( \xi_i + h(\xi_i) \right)^2 \right\}
\ee
so that
\be
\E[\omega_i^2] &= m_i - h(\xi_i) m_i (h(\xi_i) + \xi_i) + h(\xi_i) (h(\xi_i)+\xi_i)z_i - r(-\xi_i)(r(-\xi_i)+\xi_i)z_i \\
&+ (h(\xi_i) m_i+(h(\xi_i) + r(-\xi_i)+\xi_i) z_i)^2
\ee
Since $r(-\xi_i) \to 0$ at an exponential rate as $\xi_i \to \infty$, and $r(\xi_i), h(\xi_i)$ are both bounded, the leading term is just
$\xi_i^2 z_i^2 = (w_i x)^2 z_i^2$, 
so when $\xi_i > 0$ we have
\be
\E[\omega_i^2/m_i \mid \beta] = \frac{z_i^2}{m_i} (w_i x)^2 + \bigO(\xi_i).
\ee
On the other hand if $\xi_i < 0$ then by the symmetry of the problem we have
\be
\E[\omega_i^2/m_i \mid x] = \frac{(m_i - z_i)^2}{m_i} (w_i x)^2 + \bigO(\xi_i).
\ee
Now observe that
\be
x'(W'DW)x = (Wx)' D (Wx) = \sum_i m_i (w_i x)^2.
\ee
Since 
\be
\E[V(x^*) \mid x] = \sum_i \frac{((m_i - z_i) \vee z_i)^2}{m_i} (w_i x)^2 + \frac1{m_i} \bigO(|w_i x|) < \sum_i \frac{(m_i - 1)^2}{m_i} (w_i x)^2 + \frac1{m_i} \bigO(|w_i x|)
\ee
there exists a $K> 0$ such that
\be
\E[V(x^*) \mid x] &< K + \sum_i \left( \frac{(m_i - 1)^2}{m_i} + \frac1{m_i} \right) (w_i x)^2 = K + \sum_i \left( m_i - 2 + \frac2{m_i} \right) (w_i x)^2 \\
&< K+\sum_i (m_i - 1) (w_i x)^2 = K+\sum_i m_i \frac{m_i - 1}{m_i} (w_i x)^2 \\
&= \gamma V(x) + K
\ee
for $\gamma = 1 - (\max_i m_i)^{-1}$. Note we could have picked any positive constant to subtract above, and just chose $m_i^{-1}$ for convenience. 

Now, observe that the approximating kernel uses the same conditional update for $x \mid \Omega$, and that the conditional update for $\Omega \mid x$ has identical mean and covariance, which were the only quantities appearing in the calculations above. So we obtain an identical result for $\gamma_\epsilon$. In the best case, the geometric convergence rate obtained using this bound on $\gamma$ converges to zero at least at the rate $(\max_i m_i)^{-1}$. This is broadly consistent with the results in \cite{johndrow2016mcmc}, which used conductance bounds.

\subsection{Proof of Remark \ref{rem:BoundedDerivative}}
Recall
$\frac{\partial}{\partial y} \Sigma(y)^{-1} = -\Sigma^{-1} \frac{\partial 
\Sigma}{\partial y} \Sigma^{-1}$
where 
$\left( \frac{\partial \Sigma}{\partial y} \right)_{ij} = 
\frac{\partial}{\partial y} \{\Sigma(y)\}_{ij}$.
We now compute the derivative. Dependence of matrix quantities on $y$
will typically be suppressed for compactness of notation; we remind the
reader that $\Sigma, M$ are functions of $y$.
\be
\frac{\partial}{\partial y} \kappa(x,y) 
&= \{L(z \mid x,W)\}^{-1} \frac{\partial}{\partial y} |M|^{-1/2}\{b+z'M^{-1}z\}^{-\frac{a+N}2}.
\ee
Defining $D = D(y)$ as the $N \times N$ matrix with entries
$D_{ij} = -\|w_i - w_j\|^2 e^{-y \|w_i - w_j\|^2} = -\|w_i - w_j\|^2 \Sigma_{ij}$, 
we have
\be
 \frac{\partial}{\partial y} |M|^{-1/2} &= \frac12 |M|^{1/2} \tr( M D ) \\
 \frac{\partial}{\partial y} \{b+z'M^{-1}z\}^{-\frac{a+N}2} &= \frac{a+N}2 \{b + z'M^{-1} z\}^{-\frac{a+N+2}2} z'\{M^{-1} D M^{-1} \} z.
\ee
Observe that
\begin{multline*}
\left| \frac{\partial}{\partial y} \kappa(x,y) \right| \le 
|M(x)|^{1/2} \{b + z'M(x)^{-1}z\}^{\frac{a+N}2} \\
\times \bigg( |M(y)|^{-1/2} \frac{a+N}2 \{b + z'M(y)^{-1} z\}^{-\frac{a+N+2}2} |z'\{M(y)^{-1} D(y) M(y)^{-1} \} z| \\
+ \{b + z'M(y)^{-1}z\}^{-\frac{a+N}2} \frac12 |M(y)|^{1/2} |\tr\{ M(y) D(y) \}| \bigg).   
\end{multline*}

We would like to bound this uniformly away from $\infty$. In addition to the bounds in
\eqref{eq:EigenBounds}, we will also need a bound on the norm of $D$. Observe that
\be
\|D(y)\|_F^2 = \sum_{i=1}^N \sum_{j=1}^N \|w_i-w_j\|^2 e^{-y 
\|w_i-w_j\|^2} \le \sum_{i=1}^N \sum_{j=1}^N \|w_i-w_j\|^2 \equiv \bar D^2.
\ee
It follows that
$\lambda_{\max}(D(y)) \le \bar D$, 
and therefore applying standard inequalities for products of Hermitian matrices 
and quadratic forms we have
\be
\left| \frac{\partial}{\partial y} \kappa(x,y) \right| &\le 
2^{\frac{N}2} (b+\|z\|_2^2)^{\frac{a+N}2} b^{-\frac{a+N+2}2} \left( \frac{a+N}2 \bar D \|z\|_2^2 + 2^{\frac{N}2}  N \bar D  \right)
\equiv C_1,
\ee
so the derivative is uniformly bounded.

\bibliographystyle{apalike}
\bibliography{amcmc}

\end{document}